\documentclass[final]{ctextemp_siamltex}
\usepackage{amsfonts}
\usepackage{mathrsfs}
\usepackage{graphicx}
\usepackage{subcaption}
\usepackage{graphicx}
\usepackage{placeins}
\usepackage{float}
\usepackage{color}
\usepackage{CJK}
\usepackage{amsmath}

\newtheorem{remark}{\indent Remark}
\newtheorem{algorithm}{Algorithm}[section]
\newcommand*{\defeq}
{\mathrel{\vcenter{\baselineskip0.5ex \lineskiplimit0pt
                     \hbox{\scriptsize.}\hbox{\scriptsize.}}}
                     =}
\DeclareMathOperator*{\argmax}{arg\,max}
\renewcommand{\appendixname}{Appendix}

\input{tcilatex}
\begin{document}

\title{A time-accurate, adaptive discretization for fluid flow problems}
\author{Victor DeCaria\thanks{%
Department of Mathematics, University of Pittsburgh, Pittsburgh, PA 15260,
USA. (\texttt{vpd7@pitt.edu}). This author was supported by NSF grants DMS
1522267, 1817542 and CBET 160910.} \and William Layton\thanks{%
Department of Mathematics, University of Pittsburgh, Pittsburgh, PA 15260,
USA. (\texttt{wjl@pitt.edu, www.math.pitt.edu\symbol{126}wjl}). This author
was supported by NSF grants DMS 1522267, 1817542, and CBET 160910.} \and %
Haiyun Zhao\thanks{%
Department of Mathematics, University of Pittsburgh, Pittsburgh, PA 15260,
USA. (\texttt{haz50@pitt.edu}). This author was supported by NSF grants DMS
1522267, 1817542 and CBET 160910.} }
\maketitle

\begin{abstract}This report presents a low computational and cognitive complexity, stable,
time accurate and adaptive method for the Navier-Stokes equations. The
improved method requires a minimally intrusive modification to an existing
program based on the fully implicit / backward Euler time discretization,
does not add to the computational complexity, and is conceptually simple.
The backward Euler approximation is simply post-processed with a two-step,
linear time filter. The time filter additionally removes the overdamping of
Backward Euler while remaining unconditionally energy stable, proven herein.
Even for constant stepsizes, the method does not reduce to a standard /
named time stepping method but is related to a known 2-parameter family
of A-stable, two step, second order methods. Numerical tests confirm the
predicted convergence rates and the improved predictions of flow quantities
such as drag and lift.
\end{abstract}

\maketitle

\section{Introduction}

The backward Euler 
time discretization is often used for complex, viscous flows due to its
stability, rapid convergence to steady state solutions and simplicity to
implement. However, it has poor time transient flow accuracy, \cite{gresho},
and can fail by overdamping a solution's dynamic behavior. For ODEs, adding
a time filter to backward Euler, as in (1.3) below, yields two, embedded,
A-stable approximations of first and second order accuracy, \cite{guzel}.
This report develops this idea into an adaptive time-step and adaptive order
method for time accurate fluid flow simulation and gives an analysis of the
resulting methods properties for constant time-steps. For constant
time-steps, the resulting Algorithm 1.1 below involves adding only 1 extra
line to a backward Euler code. The added filter step increases accuracy and
adds negligible additional computational complexity, see Figure \ref%
{fig:timing} and Figure \ref{fig:mesh}. Further, both time adaptivity and
order adaptivity, presented in Section 2 and tested in Section 6, are easily
implemented in a constant time step backward Euler code with $\mathcal{O}%
(20) $ added lines. Thus, algorithms herein have two main features. First,
they can be implemented in a legacy code based on backward Euler without
modifying the legacy components. Second, both time step and method order can
easily be adapted due to the embedded structure of the method. The variable
step, variable order step (VSVO) method is presented in Section 2 and tested
in Section \ref{sec:test_adaptive}.

Even for constant time-steps and constant order, the method herein does not
reduce to a standard / named method. Algorithm \ref{alg:themethod} with Option B is (for constant order and
time-step) equivalent to a member of the known, 2 parameter family of second
order, 2-step, A-stable one leg methods (OLMs), see Algorithm \ref%
{alg:equiv_discrete}, Section \ref{section2}. Stability and velocity
convergence of the (constant time step) general second order, two-step,
A-stable method for the Navier-Stokes equations was proven already in 
\cite{girault}, see equation (3.20) p. 185, and has been elaborated thereafter,
e.g., \cite{JIANG2016388}. Our \textit{velocit}y stability and error analysis,
while necessary for completeness, parallels this previous work and is thus
collected in Appendix \ref{app:velocity}. On the other hand, Algorithm \ref{alg:themethod} with Option A does \emph{not} fit within a
general theory even for constant stepsize, and produces more accurate pressure approximations.

We begin by presenting the simplest, constant stepsize case to fix ideas.
\ Consider the time dependent incompressible Navier-Stokes (NS) equations: 
\begin{equation}
\begin{aligned} & u_{t}+ u\cdot\nabla u-\nu\Delta u+\nabla p = f,\;
\mathrm{and} \; \nabla\cdot u=0\; \mathrm{in} \;\Omega,\\ & u=0\;
\mathrm{on} \;\partial\Omega,\; \mathrm{and} \; \int_{\Omega} p\,d x = 0,\\
& u( x,0)= u_{0}( x)\; \mathrm{in} \;\Omega.\\ \end{aligned}  \label{NSE}
\end{equation}%
Here, $\Omega \subset \mathbb{R}^{d}$($d$=2,3) is a bounded polyhedral
domain; $u:\Omega \times \lbrack 0,T]\rightarrow \mathbb{R}^{d}$ is the
fluid velocity; $p:\Omega \times (0,T]\rightarrow \mathbb{R}$ is the fluid
pressure. The body force $f(x,t)$ is known, and $\nu $ is the kinematic
viscosity of the fluid.

Suppressing the spacial discretization, the method calculates an
intermediate velocity $\hat{u}^{n+1}$ using the
backward Euler / fully implicit method. Time filters (requiring only two
additional lines of code and not affecting the BE calculation) are applied
to produce $u^{n+1}$ and $p^{n+1}$ follows:

\begin{algorithm}[Constant $\triangle t$ BE plus time filter]
\label{alg:themethod} With $u^*=\hat{u}^{n+1}$ (Implicit) or $u^*=2u^n-u^{n-1}$ (Linearly-Implicit),
Step 1: (Backward Euler) 
\begin{equation}
\begin{aligned} &\frac{ \hat{u}^{n+1}- u^{n}}{\Delta t} +u^*\cdot
\nabla \hat{u}^{n+1} -\nu\Delta \hat{u}^{n+1} +\nabla \hat{p}^{n+1} =
f(t^{n+1}),\\ &\nabla\cdot \hat{u}^{n+1}=0, \\ \end{aligned}  \label{stp1}
\end{equation}
Step 2: (Time Filter for velocity and pressure) 

\begin{equation}  \label{stp2}
u^{n+1} = \hat{u}^{n+1}-\frac{1}{3} ( \hat{u}^{n+1}-2
u^{n}+ u^{n-1})
\end{equation}
Option A: (No pressure filter)
\begin{equation} \notag
p^{n+1} = \hat{p}^{n+1}.
\end{equation}
Option B:
\begin{equation} \notag
p^{n+1} = \hat{p}^{n+1}-\frac{1}{3} ( \hat{p}^{n+1}-2
p^{n}+ p^{n-1})
\end{equation}
Algorithm \ref{alg:themethod}A means Option A is used, and Algorithm \ref{alg:themethod}B means Option B is used.
\end{algorithm}

\begin{figure}[tbp]
\centering
\begin{subfigure}{0.49\linewidth}
   \centering
   \includegraphics[width = 1\linewidth]{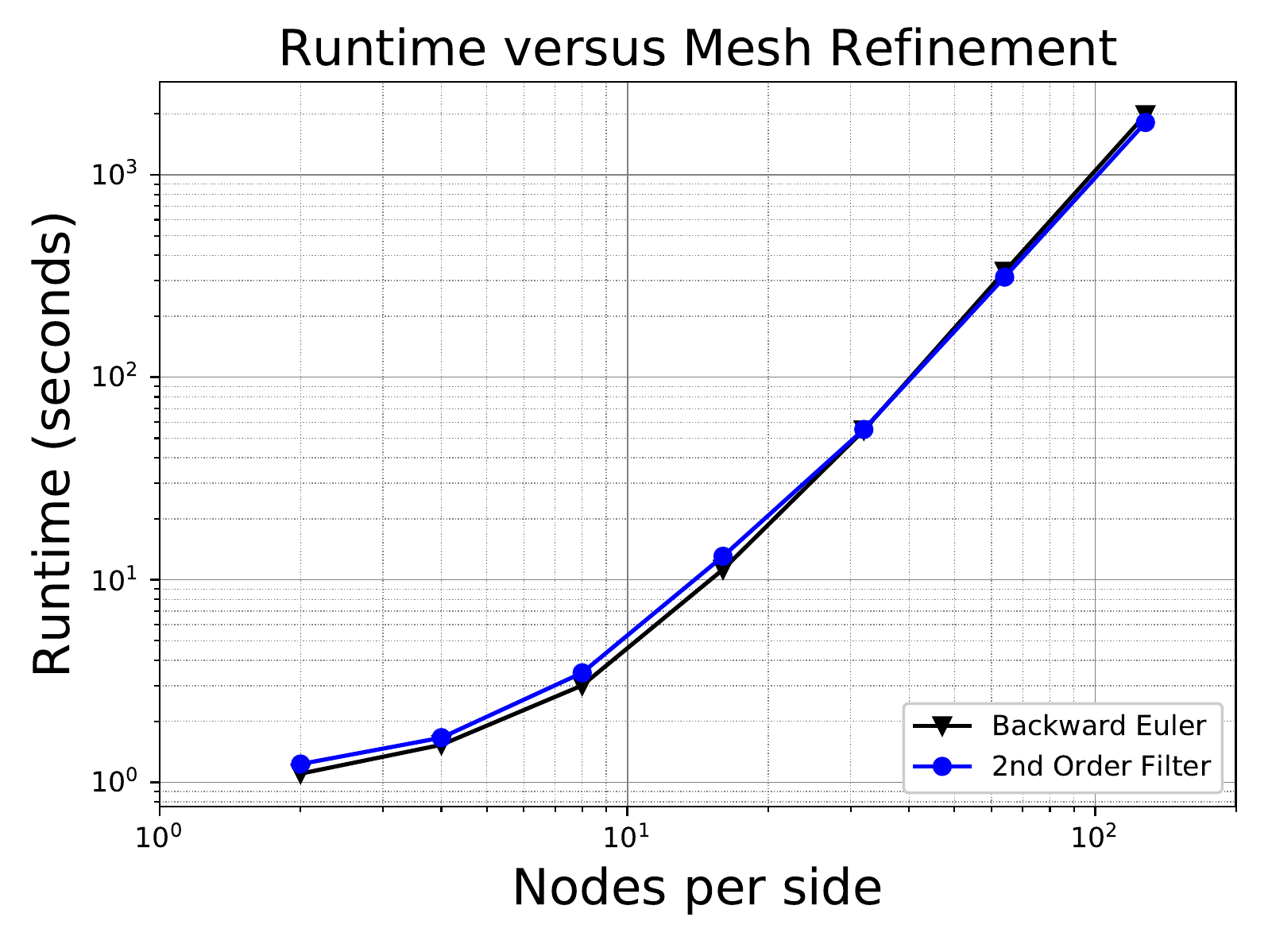}
   \caption{\label{fig:timing}}
\end{subfigure}
\begin{subfigure}{0.49\linewidth}
   \centering
   \includegraphics[width = 1\linewidth]{{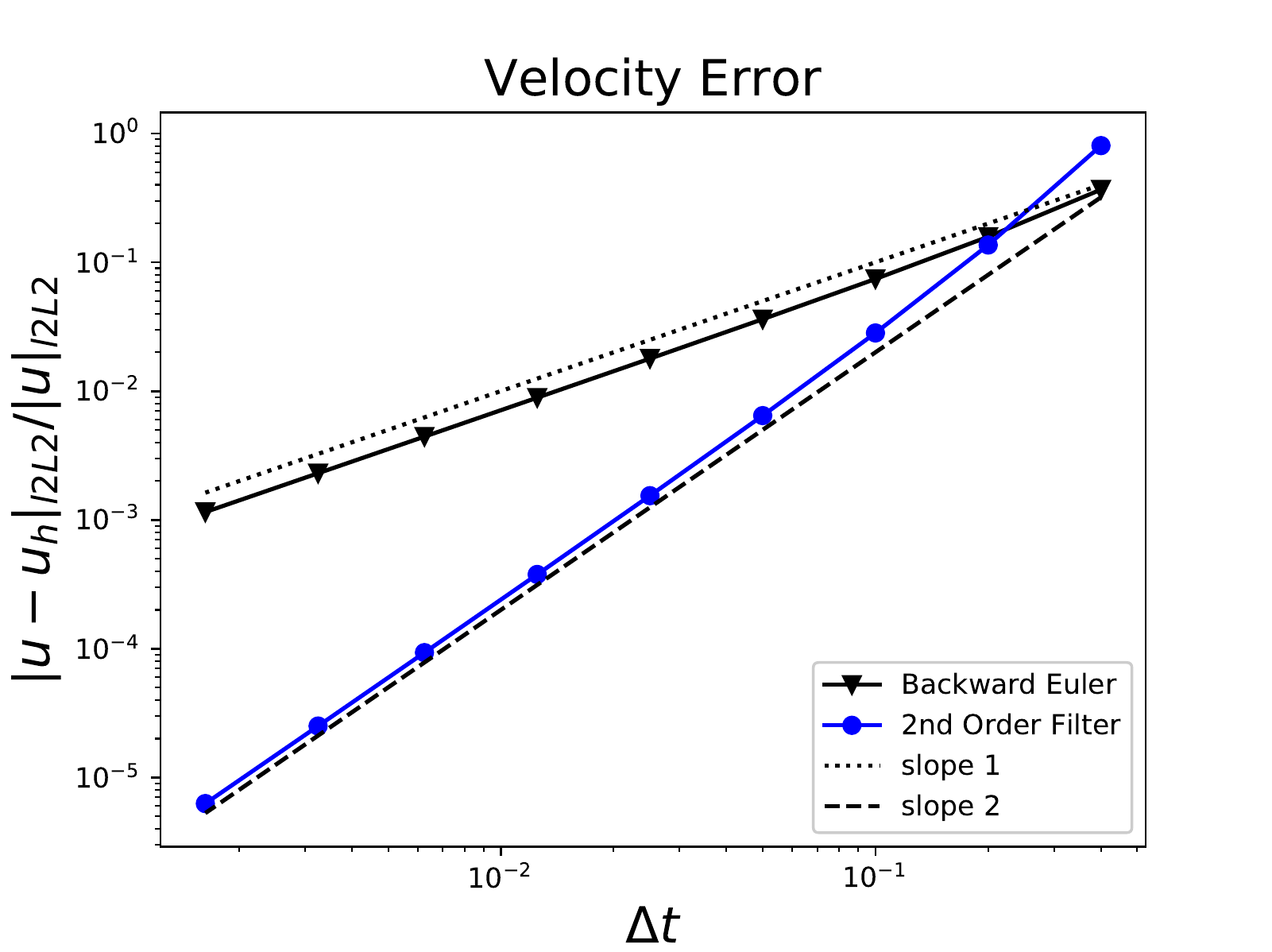}}
   \caption{\label{fig:mesh}}
\end{subfigure}
\caption{The time filter does not add to the computational complexity (Fig. 
\protect\ref{fig:timing}), yet increases the method to second order (Fig. 
\protect\ref{fig:mesh}). }
\label{fig:runtime}
\end{figure}

Its implementation in a backward Euler code does not require additional
function evaluations or solves, only a minor increase in floating point
operations. Figure \ref{fig:timing} presents a runtime comparison with and
without the filter step. It is apparent that the added computational
complexity of Step 2 is negligible. However, adding the time filter step has
a profound impact on solution quality, see Figure \ref{fig:mesh}.

Herein, we give a velocity stability and error analysis for constant
timestep in Appendix \ref{app:velocity}. Since (eliminating the intermediate step) the
constant time-step method is equivalent to an A-stable, second order, two
step method, its velocity analysis has only minor deviations from the
analysis in \cite{girault} and \cite{JIANG2016388}. We also give an analysis of the unfiltered pressure error, which does not have a
parallel in \cite{girault} or \cite{JIANG2016388}. The predicted (optimal)
convergence rates are confirmed in numerical tests in Section \ref{sec:tests}. We prove
the pressure approximation is stable and second order accurate provided only
the velocity is filtered. The predicted second order pressure convergence,
with or without filtering the pressure, is also confirmed in our tests,
Figure \ref{fig:convergence}.

The rest of the paper is organized as follow. In Section 2, we give the
full, self-adaptive VSVO algorithm for a general initial value problem.
Section 3 introduces some important mathematical notations and preliminaries
necessary and analyze the method for the Navier-Stokes equations. In Section
4, we prove unconditional, nonlinear energy stability in Theorem \ref{energy
stability}. We analyze consistency error in Section \ref{sec:consistency}. In \ref{sec:velocity_error}, we
prove $\mathcal{O}(\Delta t^{2})$ convergence for velocity, Theorem \ref%
{convergency}. The proof of the stability of the pressure is in Theorem \ref%
{thm_pressure_stability} in Section \ref{sec:pressure_stability}. We prove
second order accuracy for pressure in Section \ref{sec:pressure_error}.
Numerical tests are given in Section \ref{sec:tests} to validate the theoretical
predictions.

\subsection{Related work}

Time filters are primarily used to stabilize leapfrog time discretizations
of weather models; see \cite{robert}, \cite{asselin}, \cite{Williams2009}.
In \cite{guzel} it was shown that the time filter used herein increases accuracy to second order, preserves A-stability,
anti-diffuses the backward Euler approximation and yields an error estimator
useful for time adaptivity. The analysis in \cite{guzel} is an application
of classical numerical ODE theory and does not extend to the Navier-Stokes
equations. For the constant time step case, our analysis is based on
eliminating the intermediate approximation $\hat{u}^{n+1}$\ and reducing the
method to an equivalent two step, OLM (a twin of a linear multistep method). The velocity stability and convergence of the general A-stable
OLM was analyzed for the NSE (semi-implicit, constant time step and without
space discretization) in \cite{girault}. Thus, the constant time step, discrete 
\textit{velocity} results herein follow from these results. There is
considerable previous work on analysis of multistep time discretizations of
various PDEs, e.g. Crouzeix and Raviart \cite{crouzeix1976}. Baker,
Dougalis, and Karakashian \cite{bakerBdfthree82} gave a long-time error
analysis of the BDF methods for the NSE under a small data condition. (We
stress that the method herein is \textit{not a BDF method}.) The analysis of
the method in Girault and Raviart \cite{girault} was extended to include
spacial discretizations in \cite{JIANG2016388}. The work in \cite{JIANG2016388} also
shows how to choose those parameters to improve accuracy in higher Reynolds
number flows - a significant contribution by itself. Other interesting
extensions include the work of Gevici \cite{geveci}, Emmrich \cite%
{Emmrich2004}, \cite{Emmrich2004-2}, Jiang \cite{jiang2017}, Ravindran \cite%
{ravindran2016} and \cite{layton2014}.

\section{The Adaptive VSVO Method}

Section \ref{sec:test_adaptive} tests both the constant time step method and the method with
adaptive step and adaptive order. This section will present the algorithmic
details of adapting both the order and time step based on estimates of local
truncation errors based on established methods \cite{GH10}. The constant time step Algorithm 1.1 involves adding one
(Option A) or two (Option B) lines to
a backward Euler FEM code. The full self adaptive VSVO Algorithm 2.1 below adds $\mathcal{O%
}(20)$ lines. We first give the method for the initial value problem%
\begin{equation*}
y^{\,\,\prime }(t)=f(t,y(t)),\text{ for }t>0\text{ and }y(0)=y_{0}.
\end{equation*}
Denote the $n^{th}$ time step size by $\Delta t_n$. Let $t^{n+1}=t^{n}+\Delta t_n$ and $y^{n}$ an approximation to $y(t_{n})$. The choice of filtering weights
depend on $\omega_n \defeq\Delta t_n/\Delta t_{n-1}$, Step 2 below. $TOL$ is the user supplied tolerance on the allowable error per step.%
\begin{algorithm}[Variable Stepsize, Variable Order 1 and 2 (VSVO-12)\label{alg:vsvo_themethod}]\hfill

\noindent
\textbf{Step 1 : Backward Euler}
\begin{equation}\notag
\frac{y_{(1)}^{n+1}-y^{n}}{\Delta t_{n}}=f(t_{n+1},y_{(1)}^{n+1})
\end{equation}

\noindent
\textbf{Step 2 : Time Filter}
\begin{equation}\notag
y_{(2)}^{n+1}=y_{(1)}^{n+1}-\frac{\omega_{n+1}}{2\omega_{n+1}+1}\left(y_{(1)}^{n+1} - (1+\omega_{n+1})y^{n} + \omega_{n+1}y^{n-1}) \right)
\end{equation}

\noindent
\textbf{Step 3 : Estimate error in $y_{(1)}^{n+1}$ and $y_{(2)}^{n+1}$.}
\begin{equation}\notag
EST_1=y_{(2)}^{n+1}-y_{(1)}^{n+1}
\end{equation} 
\begin{gather*}
EST_2=\frac{\omega_{n} \omega_{n+1} (1+\omega_{n+1})}{1+2 \omega_{n+1}+\omega_{n} \left(1+4 \omega_{n+1}+3 \omega_{n+1}^2\right)}\bigg(y_{(2)}^{n+1}
\\
-\frac{(1+\omega_{n+1}) (1+\omega_{n} (1+\omega_{n+1}))}{1+\omega_{n}}y^{n} +\omega_{n+1} (1+\omega_{n} (1+\omega_{n+1}))y^{n-1}
 \\
 -\frac{\omega_{n}^2 \omega_{n+1} (1+\omega_{n+1})}{1+\omega_{n}}y^{n-2}\bigg).
\end{gather*}

\noindent
\textbf{Step 4 : Check if tolerance is satisfied.}

If $\|EST_1\|<TOL$ or $\|EST_2\|<TOL$, at least one approximation is acceptable. Go to Step 5a. Otherwise, the step is rejected. Go to Step 5b. 

\noindent
\textbf{Step 5a : At least one approximation is accepted. Pick an order and stepsize to proceed.}

If both approximations are acceptable, set 
\begin{equation}\notag
\Delta t_{(1)} = 0.9\Delta t_n \left(\frac{TOL}{\|EST_1\|}\right)^{\frac{1}{2}}, \hspace{15mm}\Delta t_{(2)}= 0.9\Delta t^n \left(\frac{TOL}{\|EST_2\|}\right)^{\frac{1}{3}}.
\end{equation}

Set $$i=\argmax_{i\in \{1,2\}} \Delta t^{(i)}, \hspace{10mm}\Delta t^{n+1}=\Delta t^{(i)}, \hspace{10mm} t^{n+2}=t^{n+1}+\Delta t_{n+1}, \hspace{ 10mm} y^{n+1} = y_{(i)}^{n+1} .$$

If only $y^{(1)}$ (resp. $y^{(2)}$) satisfies $TOL$, set $\Delta t_{n+1}=\Delta t^{(1)}$ (resp. $\Delta t^{(2)}$), and $y^{n+1} = y_{(1)}^{n+1}$ (resp. $y_{(2)}^{n+1}$). Proceed to Step 1 to calculate $y^{n+2}$.

\noindent
\textbf{Step 5b : Neither approximations satisfy TOL.}

Set 
\begin{equation}\notag
\Delta t^{(1)} = 0.7\Delta t_n \left(\frac{TOL}{\|EST_1\|}\right)^{\frac{1}{2}}, \hspace{15mm}\Delta t^{(2)}= 0.7\Delta t_n \left(\frac{TOL}{\|EST_2\|}\right)^{\frac{1}{3}}.
\end{equation}

Set $$i=\argmax_{i\in \{1,2\}} \Delta t^{(i)}, \hspace{10mm}\Delta t_{n}=\Delta t^{(i)}, \hspace{10mm} t^{n+1}=t^{n}+\Delta t_{n}$$

Return to Step 1 to try again.

\end{algorithm}
For clarity, we have not mentioned several standard features such as setting a maximum and minimum timestep, the maximum or minimum stepsize ratio, etc.

The implementation above computes an estimation of the local
errors in Step 3. $EST_1$ provides an estimation for the local error of the first order approximation $y^{(1)}_{n+1}$ since $y^{(2)}_{n+1}$ is a second order approximation. For a justification of $EST_2$, see \ref{sec:num_ode}. The optimal next stepsizes for both approximations are predicted in a standard
way in Steps 5a and 5b.  The method order (first or second) is adapted by
accepting whichever approximation satisfies the error tolerance criterion (Step 4)
and yields the larger next time step by the choice of $i=\argmax \Delta t^{(i)}$.

 Standard formulas, see e.g. \cite{griffiths}, are used to pick the next stepsize. The numbers 0.9 in Step 5a and 0.7 in Step 5b are commonly used safety factors to make the next approximation more likely to be accepted.

\textbf{One more line is needed for linearly implicit methods.} For linearly
implicit methods the point of linearization must also have $\mathcal{O}%
(\Delta t^{2})$ accuracy. For example, with $u^{\ast }=u^{n}$ 
\begin{equation}
\frac{u^{n+1}-u^{n}}{\Delta t_n}+u^{\ast }\cdot \nabla u^{n+1}+{\frac{1}{2}}(\nabla
\cdot u^{\ast })u^{n+1}+\nabla p^{n+1}-\nu \Delta u^{n+1}=f^{n+1}\text{ \& }%
\nabla \cdot u^{n+1}=0  \label{eq:standardBEforNSE}
\end{equation}%
is a common first order linearly implicit method. The required modification in the BE
step to ensure second order accuracy after the filter is to shift the point
of linearization from $u^{\ast }=u^{n}$ to%
\begin{equation}
u^{\ast }=\left( 1+\frac{\Delta t_{n+1}}{\Delta t_{n}}\right) u^{n}-\frac{\Delta t_{n+1}}{\Delta t_{n}} u^{n-1} = 
\left( 1+\omega_n\right) u^{n}-\omega_n u^{n-1}.  \notag
\end{equation}

\textbf{Other simplifications.} The algorithm can be simplified if only the
time-step is adapted (not order adaptive). It can be further simplified using \emph{extrapolation} where the second order approximation is adapted based on $EST_1$
(pessimistic for the second order approximation). 

\section{Notations and preliminaries}

\label{section2} We introduce some notations and inequalities which will be
used in later sections. $(\cdot ,\cdot ),\Vert \cdot \Vert $ denotes the $%
L^{2}(\Omega )$ inner product and norm. $C$ will denote a generic, finite
constant depending possibly on $T$, $\Omega $ and $f$. The velocity space $X$
and pressure space $Q$ are defined 
\begin{equation*}
\begin{aligned} & X:=H^{1}_{0}(\Omega)^{d}=\{v\in
H^{1}(\Omega)^{d}:v|_{\partial\Omega}=0\},\\ &
Q:=L^{2}_{0}(\Omega)^{d}=\{q\in L^{2}(\Omega):\int_{\Omega}q=0\}.
\end{aligned}
\end{equation*}%
The divergence free space $V$ is given by 
\begin{equation*}
\begin{aligned} & V:=\{v\in X:(\nabla\cdot v,q)=0\quad\forall q\in Q)\}.
\end{aligned}
\end{equation*}%
For measurable $v:[0,T]\rightarrow X$, define for, respectively, $1\leq
p<\infty $ and $p=\infty $ 
\begin{equation*}
||v||_{L^{p}(0,T;X)}=\left( \int_{0}^{T}||v(t)||_{X}^{p}dt\right) ^{1/p}%
\text{ and }||v||_{L^{\infty }(0,T;X)}=ess\sup_{0\leq t\leq
T}||v(t)||_{X}^{p}.
\end{equation*}
We define the skew-symmetrized nonlinear form: 
\begin{equation*}
\begin{aligned} & B( u, v):= u\cdot\nabla v+\frac{1}{2}(\nabla\cdot u)
v,\quad \forall\; u, v, w \in X,\\ &b( u, v, w):=(B( u, v), w).\\
\end{aligned}
\end{equation*}

\begin{lemma}\label{nonlinear_inequality}
There exists $C >0$ such that 
\begin{equation}\notag
\begin{aligned}
&b( u, v, w)\leq C\|\nabla u\|\|\nabla v\|\|\nabla w\|, \quad\quad \forall\; u, v, w\in X \\
&b( u, v, w)\leq C\| u\|\| v\|_{2}\|\nabla w\|\quad \forall u, w\in X, v\in X\cap H^{2}(\Omega).
\end{aligned}
\end{equation}
\end{lemma}%
\begin{proof}
See Lemma 2.1 on p. 12 of \cite{temamNonlinear}.
\end{proof}

We use the following discrete Gronwall inequality found in \cite[Lemma 5.1]%
{heywood1990}. 
\begin{lemma}[Discrete Gronwall Inequality]\label{discrete_gronwall}
Let $\Delta t$, $H$, $a_n,b_n,c_n,d_n$ (for integers $n\geq0$) be non-negative numbers such that
\begin{equation}
a_l +\Delta t \sum_{n=0}^l b_n \leq \Delta t \sum_{n=0}^l d_na_n+\Delta t \sum_{n=0}^l c_n + H, \quad \forall\; l \geq 0
\end{equation}
Suppose $\Delta td_n < 1$ $ \forall n$, then,
\begin{equation}
a_l + \Delta t \sum^l_{n=0}b_n \leq \exp \Big( \Delta t \sum^l_{n=0}b_n \frac{d_n}{1-\Delta t d_n} \Big) \Big(\Delta t \sum^l_{n=0} c_n + H \Big), \quad \forall\; l \geq 0
\end{equation}
\end{lemma}

Multiplying \eqref{NSE} by test functions $(v,q)\in (X,Q)$ and integrating
by parts gives 
\begin{equation}
(u_{t},v)+b(u,u,v)+\nu (\nabla u,\nabla v)-(p,\nabla \cdot v)+(\nabla \cdot
u,q)=(f,v), \hspace{5mm}(\nabla \cdot u,q)=0.  \label{eqn:nse_var_form}
\end{equation}%
To discretize the above system in space, we choose conforming finite element
spaces for velocity $X^{h}\subset X$ and pressure $Q^{h}\subset Q$
satisfying the discrete inf-sup condition and the following approximation
properties: 
\begin{equation}
\begin{aligned} &\inf_{q_h \in Q_h} \sup_{v_h \in
X_h}\dfrac{(q_h,\nabla\cdot v_h)}{\|q_h\|\|\nabla v_h\|} \geq \beta > 0, \\
&\inf_{v \in X^h} \| u-v\| \leq Ch^{k+1}\|u\|_{k+1}, \quad u \in
H^{k+1}(\Omega)^d \\ &\inf_{v \in X^h} \| u-v\|_1 \leq Ch^{k+1}\|u\|_{k},
\quad u \in H^{k+1}(\Omega)^d \\ &\inf_{r \in Q^h} \| p-r\| \leq
Ch^{s+1}\|p\|_{s+1}, \quad p \in H^{s+1}(\Omega) \\ \end{aligned}
\end{equation}%
$h$ denotes the maximum triangle diameter.
Examples of finite element spaces satisfying these conditions are the MINI \cite{Arnold1984} and Taylor-Hood \cite{M2AN_1984__18_2_175_0} elements. The discretely divergence free
subspace $V_{h}\in X_{h}$ is defined 
\begin{equation*}
\begin{aligned} & V_{h}:=\{ v_{h}\in X_{h}:(\nabla\cdot
v_{h},q_{h})=0\quad\forall q_{h}\in Q_{h}\}. \end{aligned}
\end{equation*}%
The dual norms of $X_{h}$ and $V_{h}$ are 
\begin{equation*}
\Vert w\Vert _{X_{h}^{\ast }}:=\sup_{v_{h}\in X_{h}}\dfrac{(w,v_{h})}{\Vert
\nabla v_{h}\Vert },\Vert w\Vert _{V_{h}^{\ast }}:=\sup_{v_{h}\in V_{h}}%
\dfrac{(w,v_{h})}{\Vert \nabla v_{h}\Vert }.
\end{equation*}%
The following Lemma from Galvin \cite[p. 243]{galvin} establishes the
equivalence of these norms on $V_{h}$. 
\begin{lemma}\label{dual_inequality}
Suppose the discrete inf-sup condition holds, let $w \in V_h$, then there exists $C >0$, independent of $h$, such that 
\begin{equation}\notag
C\|w\|_{X^*_h} \leq \|w\|_{V^*_h} \leq \|w\|_{X^*_h}.
\end{equation}
\end{lemma}Lemma \ref{dual_inequality} is used to derive pressure error
estimates with a technique shown in Fiordilino \cite{fiordilino_pres}. We
will use the following, easily proven, algebraic identity. 
\begin{lemma}\label{identity}
The following identity holds.
\begin{gather}\label{eqn:theidentity_part2}
\left(\frac{3}{2}a-2b+\frac{1}{2}c\right)\left(\frac{3}{2}a-b+\frac{1}{2}c\right) = \\
\left(\frac{a^2}{4}+\frac{(2a-b)^2}{4}+\frac{(a-b)^2}{4}\right)-\left(\frac{b^2}{4}+\frac{(2b-c)^2}{4}+\frac{(b-c)^2}{4}\right)
+\frac{3}{4}(a-2b+c)^2 \notag
\end{gather}
\end{lemma}

With the notation in place, we state the fully discrete method.

\begin{algorithm}[Fully Discrete Method\label{alg:discrete}]
Given $u_{h}^{n-1},u_{h}^{n}\in X^{h}$ (and if necessary, given $p_{h}^{n-1},p_{h}^{n}\in Q^{h}$),  find $(\hat{u}_{h}^{n+1},\hat{p}%
^{n+1})\in (X^{h},Q^{h})$ satisfying 
\begin{align}
\left( \frac{\hat{u}_{h}^{n+1}-u_{h}^{n}}{\Delta t_{n}},v_{h}\right) +b(\hat{u}%
_{h}^{n+1},\hat{u}_{h}^{n+1},v_{h})+\nu (\nabla \hat{u}_{h}^{n+1},\nabla
v_{h})-(\hat{p}_{h}^{n+1},\nabla \cdot v_{h})& =(f(t^{n+1}),v_{h}),
\label{eqn:discrete} \\
(\nabla \cdot \hat{u}^{n+1},q_{h})& =0.  \notag
\end{align}%
for all $(v_{h},q_{h})\in (X^{h},Q^{h})$. Then compute 
\begin{gather*}
u_{h}^{n+1}=\hat{u}_{h}^{n+1}-\frac{\omega_{n+1}}{2\omega_{n+1}+1}\left(\hat{u}_{h}^{n+1} - (1+\omega_{n+1})u_{h}^{n} + \omega_{n+1}u_{h}^{n-1}) \right).
\end{gather*}
Option A: (No pressure filter)
\begin{equation} \notag
p_h^{n+1} = \hat{p}_h^{n+1}.
\end{equation}
Option B:
\begin{equation} \notag
p_{h}^{n+1}=\hat{p}_{h}^{n+1}-\frac{\omega_{n+1}}{2\omega_{n+1}+1}\left(\hat{p}_{h}^{n+1} - (1+\omega_{n+1})p_{h}^{n} + \omega_{n+1}p_{h}^{n-1}) \right).
\end{equation}

\end{algorithm}

The constant time-step stability and error analysis works with the following
equivalent formulation of the method. We stress that what follows is not the
preferred implementation since it only yields one approximation, while
Algorithm \ref{alg:discrete} gives the embedded approximations $\hat{u}%
_{h}^{n+1}$ and $u_{h}^{n+1}$ and an error
estimator.

\begin{algorithm}[Constant time-step, equivalent method\label%
{alg:equiv_discrete}]
Assume the time-step is constant. Given $(u_{h}^{n},p_{h}^{n})$ and $%
(u_{h}^{n-1},p_{h}^{n-1})$, find $(u_{h}^{n+1},p_{h}^{n+1})$ such that for all $(v_h,q_h) \in (X^h,Q^h)$,

\noindent
\textbf{Option A}
\begin{gather}\label{eqn:equiv_method_full}
\left(\frac{\frac{3}{2} u^{n+1}_{h}-2 u^{n}_{h}+\frac{1}{2}
u^{n-1}_{h}}{\Delta t}, v_{h}\right) +b \left( \frac{3}{2} u^{n+1}_{h}-
u^{n}_{h}+\frac{1}{2} u^{n-1}_{h},\frac{3}{2} u^{n+1}_{h}-
u^{n}_{h}+\frac{1}{2} u^{n-1}_{h}, v_{h} \right)
\\
 +\nu \left( \nabla
\left(\frac{3}{2} u^{n+1}_{h}- u^{n}_{h}+\frac{1}{2} u^{n-1}_{h}\right),
\nabla v_{h}\right) -\mathbf{\left(p_h^{n+1}, \nabla\cdot v_{h}\right)} = \left( f^{n+1},
v_{h}\right),\notag
\\
 \left(\nabla\cdot \left(\frac{3}{2} u^{n+1}_{h}-
u^{n}_{h}+\frac{1}{2} u^{n-1}_{h}\right),q_{h}\right)=0,\notag
\end{gather}

\noindent
\textbf{or Option B}
\begin{gather}\label{eqn:equiv_method_full_optb}
\left(\frac{\frac{3}{2} u^{n+1}_{h}-2 u^{n}_{h}+\frac{1}{2}
u^{n-1}_{h}}{\Delta t}, v_{h}\right) +b \left( \frac{3}{2} u^{n+1}_{h}-
u^{n}_{h}+\frac{1}{2} u^{n-1}_{h},\frac{3}{2} u^{n+1}_{h}-
u^{n}_{h}+\frac{1}{2} u^{n-1}_{h}, v_{h} \right)
\\
 +\nu \left( \nabla
\left(\frac{3}{2} u^{n+1}_{h}- u^{n}_{h}+\frac{1}{2} u^{n-1}_{h}\right),
\nabla v_{h}\right) -\mathbf{\left(\frac{3}{2} p^{n+1}_{h}-p^{n}_{h}+\frac{1}{2} p^{n-1}_{h}, \nabla\cdot v_{h}\right)} = \left( f^{n+1},
v_{h}\right),\notag
\\
 \left(\nabla\cdot \left(\frac{3}{2} u^{n+1}_{h}-
u^{n}_{h}+\frac{1}{2} u^{n-1}_{h}\right),q_{h}\right)=0.\notag
\end{gather}
\end{algorithm}
The pressure is highlighted in bold, and is the only difference between the two above equations.
The time difference term of the above equivalent method is that of BDF2 but
the remainder is different. This is \textit{not} the standard BDF2 method. 
\begin{proposition}
Algorithm \ref{alg:discrete}A (respectively B) is equivalent Algorithm \ref{alg:equiv_discrete}A (respectively B).
\end{proposition}%
\begin{proof}We will just prove the case for Option A since the other case is similar.
Let $(u_h^{n+1},p_h^{n+1})$ be the solution to Algorithm \ref{alg:discrete}. By linearity of the time filter, $(u_h^{n+1},p_h^{n+1}) \in (X^h,Q^h)$. We can write $\hat{u}_h^{n+1}$ in terms of $u_h^{n+1}$,$u_h^{n}$, and $u_h^{n-1}$ as
$\hat{u}^{n+1}=\frac{3}{2}u^{n+1}-u^{n}+\frac{1}{2}u^{n-1}$. Substitute this into \eqref{eqn:discrete}. Then $(u_h^{n+1},p_h^{n+1})$ satisfies equation \eqref{eqn:equiv_method_full}. 

These steps can be reversed to show the converse.
\end{proof}

We next define the discrete kinetic energy, viscous and numerical
dissipation terms that arise naturally from a G-stability analysis of
Algorithm \ref{alg:equiv_discrete}, \emph{regardless} of whether Option A or B is used. The (constant time-step) discrete
kinetic energy, discrete viscous energy dissipation rate and the numerical
energy dissipation rate of Algorithm \ref{alg:equiv_discrete} are%
\begin{equation*}
\begin{array}{cl}
\text{discrete energy:} & \mathcal{E}^{n}=  \frac{1}{4}\left[ \Vert
u^{n}\Vert ^{2}+\Vert 2u^{n}-u^{n-1}\Vert ^{2}+\Vert u^{n}-u^{n-1}\Vert ^{2}%
\right] , \\ 
\text{ viscous dissipation:} & \mathcal{D}^{n+1}=  \Delta t\nu ||\nabla
\left( \frac{3}{2}u^{n+1}-u^{n}+\frac{1}{2}u^{n-1}\right) ||^{2}, \\ 
\text{numerical dissipation:} & \mathcal{Z}^{n+1}=  \frac{3}{4}\Vert
u^{n+1}-2u^{n}+u^{n-1}\Vert ^{2} .%
\end{array}%
\end{equation*}

\begin{remark}
As $\Delta t\rightarrow 0$, $\mathcal{E}^{n}$ is consistent with the kinetic
energy $\frac{1}{2}\Vert u\Vert ^{2}$ and $\mathcal{D}^{n}$ is consistent
with the instantaneous viscous dissipation $\nu \Vert \nabla u\Vert ^{2}$.
The numerical dissipation $\mathcal{Z}^{n+1}\approx \frac{3}{4}\Delta
t^{4}\Vert {u_{tt}(t^{n+1})}\Vert ^{2}$, is asymptotically smaller than the
numerical dissipation of backward Euler, $\frac{1}{2}%
\Delta t^{2}\Vert {u_{t}(t^{n+1})}\Vert ^{2}$.

The method's kinetic energy differs from that of BDF2, which is (e.g. \cite%
{LaytonTrenchea}) 
\begin{equation*}
\mathcal{E}_{BDF2}^{n}=\frac{1}{4}\left[ \Vert u^{n}\Vert ^{2}+\Vert
2u^{n}-u^{n-1}\Vert ^{2}\right] 
\end{equation*}%
due to the term $\Vert u^{n}-u^{n-1}\Vert ^{2}$\ in $\mathcal{E}^{n}$\ which
is a dispersive penalization of a discrete acceleration.
\end{remark}

Define the interpolation and difference operators as follows 
\begin{definition}\label{def:int_dif}
The interpolation operator $I$ and difference operator $D$ are
\begin{equation}
I[w^{n+1}]=\frac{3}{2}w^{n+1}-w^{n}+\frac{1}{2}w^{n-1} \quad \text{and}
\quad D[w^{n+1}]=\frac{3}{2}w^{n+1}-2w^{n}+\frac{1}{2}w^{n-1}.  \notag
\end{equation}
\end{definition}

Formally, $I[w(t^{n+1})]=w(t^{n+1}) + \mathcal{O}(\Delta t^2)$, and $\frac{%
D[w(t^{n+1})]}{\Delta t}=w_t(t^{n+1}) + \mathcal{O}(\Delta t^2)$. This will
be made more precise in the consistency error analysis in Section \ref%
{sec:consistency}.

\section{Stability and Error Analysis}
We prove stability and error analysis of the \textit{constant time-step}
method. The velocity proofs parallel ones in \cite{girault} and \cite{JIANG2016388}
and are collected in Appendix \ref{app:velocity}. The pressure analysis is presented in
Section 5.\label{section3} 

\begin{theorem}\label{energy stability}
Assume the stepsize is constant. The following equality holds.
\begin{equation}\notag
\mathcal{E}^{N} + \sum_{n=1}^{N-1}\mathcal{D}^{n+1} +  \sum_{n=1}^{N-1}\mathcal{Z}^{n+1} = \Delta t\sum_{n=1}^{N-1}(f,I[u_h^{n+1}]) + \mathcal{E}^1.
\end{equation}
\end{theorem}%
\begin{proof}
In Algorithm \ref{alg:equiv_discrete}, set $v_h = I[u_h^{n+1}]$ and $q_h = p_h^{n+1}$ for Option A, or $q_h = I[p_h^{n+1}]$ for Option B, and add.
\begin{equation}\label{eqn:stab-innerprod}
(D[u_h^{n+1}],I[u_h^{n+1}])
 + \mathcal{D}^{n+1}
=
\Delta t( f,I[u_h^{n+1}]).
\end{equation}
By Lemma \ref{identity} and Definition \ref{def:int_dif},
\begin{equation}\notag
(D[u_h^{n+1}],I[u_h^{n+1}]) = \mathcal{E}^{n+1} - \mathcal{E}^{n} + \mathcal{Z}^{n+1}.
\end{equation}
Thus, (\ref{eqn:stab-innerprod}) can be written
\begin{equation}\notag
\mathcal{E}^{n+1} - \mathcal{E}^{n} + \mathcal{D}^{n+1} + \mathcal{Z}^{n+1} = \Delta t( f(t^{n+1}),I[u_h^{n+1}]).
\end{equation}
Summing over $n$ from $1$ to $N-1$ yields the result.
\end{proof}

This result is for the time stepping method applied to the Navier-Stokes
equations. More generally, the constant time-step method of Algorithm 3.2 is 
$G$-Stable, a fact that follows from the equivalence of $A$ and $G$%
-Stability \cite{dahlquist78}. We calculate the $G$ matrix explicitly below. 
\begin{corollary}
Assume the time-step is constant. Backward Euler followed by the time filter is $G$-Stable with $G$ matrix
\begin{equation}\notag
G = 
\begin{bmatrix}
\frac{3}{2} & -\frac{3}{4} \\
-\frac{3}{4} & \frac{1}{2}
\end{bmatrix}.
\end{equation}
\begin{proof}
Simply check that
\begin{equation}\notag
[u^n,u^{n-1}]G 
\begin{bmatrix}
u^n \\
u^{n-1}
\end{bmatrix}
= \frac{1}{4}\Big[|u^{n}|^{2}+|2u^{n}-u^{n-1}|^{2}+%
|u^{n}-u^{n-1}|^2\Big].
\end{equation}
\end{proof}
\end{corollary}

\subsection{Consistency error}

\label{sec:consistency} By manipulating \eqref{eqn:nse_var_form}, we derive
the consistency error. The true solution to \eqref{eqn:nse_var_form}
satisfies 
\begin{equation}  \label{eqn:true_with_consistency}
\begin{aligned} &\left( \frac{D[u(t^{n+1})]}{\Delta t}, v_{h}\right)
+b\left( I[u(t^{n+1})],I[u(t^{n+1})], v_{h}\right) \\ &+\nu\left( \nabla
I[u(t^{n+1})], \nabla v_{h}\right) -\left( {p}(t^{n+1}),\nabla \cdot
v_{h}\right) \\ &=\left( \mathbf{f}^{n+1}, v_{h}\right) +\tau^{n+1}(u,p;v_h)
\quad \forall v_{h} \in {X}_{h}.\\ \end{aligned}
\end{equation}
If Option A is used (pressure is unfiltered),
\begin{gather}
\tau^{n+1}(u,p;v_h) =\tau_A^{n+1}(u,p;v_h)  \defeq \left( \frac{D[u(t^{n+1})]}{\Delta t}- u_{t}{(t^{n+1})%
}, v_{h}\right)  \label{eqn:consistency_defn} \\
+b\left(I[u(t^{n+1})],I[u(t^{n+1})], v_{h}\right) \notag -
b(u(t^{n+1}),u(t^{n+1}),v_h)  \notag 
+\nu \left( \nabla (I[u(t^{n+1})] - u(t^{n+1})), \nabla v_{h}\right) \notag
\end{gather}
If Option B is used (pressure is filtered),
\begin{gather}
\tau^{n+1}(u,p;v_h) = \tau_A^{n+1}(u,p;v_h) - \left(I[p(t^{n+1})]-p(t^{n+1}),\nabla \cdot v_h\right)
  \label{eqn:consistency_defnb}
\end{gather}
Thus, filtering the pressure introduces a term that, while still second order, adds to the consistency error. We believe this is why Option A performs better in the numerical tests, Figure \ref{fig:convergence}. Furthermore, Option B requires assuming additional regularity for convergence, see Theorem \ref{convergency}.

The terms in the consistency error are bounded in the following lemma. 
\begin{lemma}[Consistency] For $u$, $p$ sufficiently smooth, we have \label{consis_lemma}
\begin{equation}\notag
\bigg\|{\frac{D[u(t^{n+1})]}{\Delta t}} - u_t(t^{n+1})\bigg\|^2 \leq \frac{6}{5}\Delta t^3\int_{t^{n-1}}^{t^{n+1}} \|u_{ttt}\|^2dt, 
\end{equation}
\begin{equation}
\bigg\|I[u(t^{n+1})] - u(t^{n+1})\bigg\|^2 \leq \frac{4}{3}\Delta t^3 \int_{t^{n-1}}^{t^{n+1}} \|u_{tt}\|^2 dt.
\end{equation}
\begin{equation}
\bigg\|I[p(t^{n+1})] - p(t^{n+1})\bigg\|^2 \leq \frac{4}{3}\Delta t^3 \int_{t^{n-1}}^{t^{n+1}} \|p_{tt}\|^2 dt.
\end{equation}
\end{lemma}

\begin{proof}
See Appendix \ref{app:velocity}.
\end{proof}

\subsection{Error estimates for the velocity}

\label{section4} Next, we analyze the convergence of Algorithm \ref%
{alg:equiv_discrete} and give an error estimate for the velocity. Let $%
t^{n}=n\Delta t$. Denote the errors $\mathbf{e}_{u}^{n}=u(t^{n})-u_{h}^{n}$
and $e_{p}^{n}=p(t^{n})-p_{h}^{n}$. 

\begin{theorem}\label{convergency}
Assume that the true solution $( u,p)$ satisfies the following regularity 
\begin{equation}\label{regularity}
\begin{aligned}
&u\in L^\infty(0,T;(H^{k+1}\Omega))^d),
&u_t \in L^2(0,T;(H^{k+1}\Omega))^d),
&\quad u_{tt}\in L^2(0,T;(H^{1}\Omega))^d),\\
&u_{ttt} \in L^2(0,T;(L^{2}\Omega))^d),
&p \in L^2(0,T;(H^{s+1}(\Omega))^d).
\end{aligned}
\end{equation}
Additionally for Option B, assume $p_{tt} \in L^2(0,T;(L^2(\Omega))^d$.
For $( u^{n+1}_h,p^{n+1}_h)$ satisfying $(\ref{eqn:equiv_method_full})$, we have the following estimate
\begin{equation}\label{velocity_error}
\begin{aligned}
&\|{e}^{N}_{ u}\|^2+\|2{e}^{N}_{ u}-{e}^{N-1}_{ u}\|^2+\|{e}^{N}_{ u}-{e}^{N-1}_{ u}\|^2+\sum_{n=1}^{N-1} 3 \|{e}^{n+1}_{ u}-2{e}^{n}_{ u}+{e}^{n-1}_{ u}\|^2 \\
&+\nu \Delta t \sum_{n=1}^{N-1} \|\nabla I[{e}^{n+1}_{u}]\|^2 \leq
C\Big(h^{2k}+h^{2s+2}+\Delta t ^4\Big)
\end{aligned}
\end{equation} 
\end{theorem}\begin{proof}
See Appendix \ref{app:velocity}.
\end{proof}

\section{Pressure Stability and Convergence}

\subsection{Stability of Pressure}

\label{sec:pressure_stability} We introduce the following discrete norms 
\begin{equation}
\begin{aligned} &\||\omega\||_{\infty,k}:=\max_{0\leq n\leq T/ \Delta t}
\|\omega^n\|_k, \quad \||\omega\||_{2,k}:=\left( \sum_{n=0}^{T/ \Delta t -1}
\Delta t \|\omega^n\|^2_k \right)^{1/2}. \end{aligned}  \label{discrete_norm}
\end{equation}%
In this section, we prove that the pressure approximation is stable in $l^{1}(0,T;L^{2}(\Omega ))$. We first
give a corollary of Theorem \ref{energy stability} asserting the stability
of the velocity approximation. 
\begin{corollary}\label{velocity_stability}
Suppose $f \in L^2(0,T;H^{-1}(\Omega)^d)$, then the velocity approximation satisfies
\begin{equation}\notag
\mathcal{E}^{N} + \frac{1}{2}\sum_{n=1}^{N-1}\mathcal{D}^{n+1} +  \sum_{n=1}^{N-1}\mathcal{Z}^{n+1} \leq  \frac{1}{2\nu}\||f\||^2_{2,-1} + \mathcal{E}^1.
\end{equation}
\end{corollary}%
\begin{proof}
Consider Theorem $\ref{energy stability}$. Applying the  Cauchy-Schwartz yields the inequality.
\end{proof}

We now prove the stability of the filtered pressure. 
\begin{theorem}\label{thm_pressure_stability}
Suppose Corollary $\ref{velocity_stability}$ holds, then the pressure approximation satisfies
\begin{equation}
\begin{aligned}
&\beta \Delta t \sum_{n=1}^{N-1} \|p_h^{n+1}\| \leq C &\text{for Option A},\\
&\beta \Delta t \sum_{n=1}^{N-1} \|I[p_h^{n+1}]\| \leq C &\text{for Option B}.
\end{aligned}
\end{equation}
\end{theorem}%
\begin{proof} We prove it for Option A, as the other case is similar.
Isolating the discrete time derivative in \eqref{eqn:equiv_method_full}, and restricting $v_h$ to $V_h$ yields
\begin{equation}\label{isolate_time_derivative}
\begin{aligned}
&\left(\frac{D[u^{n+1}_{h}]}{\Delta t}, v_{h}\right) 
=-b\left( I[u^{n+1}_{h}],I[u^{n+1}_{h}],  v_{h}\right)\\ 
&-\nu \left( \nabla I[u^{n+1}_{h}],  \nabla  v_{h}\right) 
+\left(  f^{n+1},  v_{h}\right) \quad \forall v_h \in V_h.
\end{aligned}
\end{equation}
The terms on the right hand side of (\ref{isolate_time_derivative}) can be bounded as follows,
\begin{equation}\label{pressure_estimates}
\begin{aligned}
&b\left( I[u^{n+1}_{h}],I[u^{n+1}_{h}],  v_{h}\right)
\leq 
C\|\nabla I[u^{n+1}_{h}]\| \|\nabla I[u^{n+1}_{h}]\| \|\nabla v_h\|, \\
&-\nu \left( \nabla I[u^{n+1}_{h}],  \nabla  v_{h}\right)
\leq
\nu \|\nabla I[u^{n+1}_{h}]\| \|\nabla v_h\|,\\
&\left(  f^{n+1},  v_{h}\right) 
\leq 
\|f^{n+1}\|_{-1} \|\nabla v_h\|.
\end{aligned}
\end{equation}
In equation (\ref{isolate_time_derivative}), we can use the above estimates in (\ref{pressure_estimates}), divide both sides by $\|\nabla v_h\|$, and take the supremum over $v_h \in V_h$. This gives
\begin{equation}
\begin{aligned}
&\bigg\|\dfrac{D[u^{n+1}_{h}]}{\Delta t}\bigg\|_{V^*_h} 
\leq
(C\|\nabla I[u^{n+1}_{h}]\|+ \nu)\|\nabla I[u^{n+1}_{h}]\| 
+\|f^{n+1}\|_{-1}.
\end{aligned}
\end{equation}
Lemma \ref{dual_inequality} implies 
\begin{equation}
\begin{aligned}
&\bigg\|\dfrac{D[u^{n+1}_{h}]}{\Delta t}\bigg\|_{X^*_h} 
\leq
C\Big[(\|\nabla I[u^{n+1}_{h}]\|+ 1)\|\nabla I[u^{n+1}_{h}]\| 
+\|f^{n+1}\|_{-1}\Big].
\end{aligned}
\end{equation}
Now consider Algorithm \ref{alg:equiv_discrete} again with $v_h \in X_h$. Isolating the pressure term in (\ref{eqn:equiv_method_full}) and using the estimates from (\ref{pressure_estimates}) yields
\begin{gather}\label{before_final_pressure_stability}
\left( p^{n+1}_{h}, \nabla\cdot  v_{h}\right)
\leq
\left(\frac{D[u^{n+1}_{h}]}{\Delta t}, v_{h}\right)\\  
+C(\|\nabla I[u^{n+1}_{h}]\|+ 1)\|\nabla I[u^{n+1}_{h}]\| \|\nabla v_h\|
+\|f^{n+1}\|_{-1} \|\nabla v_h\|.\notag
\end{gather}
Divide both sides by $\|\nabla v_h\|$, take supremum over $v_h \in X_h$ and use the discrete inf-sup condition and the results in (\ref{before_final_pressure_stability}). Then,
\begin{equation}
\begin{aligned}
&\beta\|p^{n+1}_{h}\| 
\leq C\Big[(\|\nabla I[u^{n+1}_{h}]\|+ 1)\|\nabla I[u^{n+1}_{h}]\| 
+\|f^{n+1}\|_{-1}\Big].
\end{aligned}
\end{equation}
We then multiply by $\Delta t$, sum from $n=1$ to $n=N-1$, and apply Cauchy-Schwartz on the right hand,
\begin{equation}
\begin{aligned}
&\beta \Delta t \sum_{n=1}^{N-1} \|p^{n+1}_{h}\| 
\leq C\Delta t\Big[(\||\nabla I[u^{n+1}_{h}]\||_{2,0}+  1)\||\nabla I[u^{n+1}_{h}]\||_{2,0} 
+\|f^{n+1}\|_{2,-1}\Big].
\end{aligned}
\end{equation}
Then using the result from velocity approximation, we get,
\begin{equation}
\begin{aligned}
&\beta \Delta t \sum_{n=1}^{N-1} \|p^{n+1}_{h}\| 
\leq C\Big[( \||f\||_{2,-1}+1)\||f\||_{2,-1} 
+( \mathcal{E}^1 +1)\mathcal{E}^1 \Big].
\end{aligned}
\end{equation}
\end{proof}

\subsection{Error estimates for the pressure\label{sec:pressure_error_analysis}}

\label{sec:pressure_error} We now prove convergence of the pressure
approximation in $l^{1}(0,T;L^{2}(\Omega ))$. Denote the pressure error as $%
e_{p}^{n}=p(t^{n})-p_{h}^{n}$. 
\begin{theorem} \label{pressure_error}
Let $u$, $p$ satisfy the equation $(\ref{velocity_error})$. 
Let the assumption of regularity in Theorem $\ref{convergency}$ be satisfied. Then there exists a constant $C>0$ such that
\begin{equation}
\begin{aligned}
&\Delta t \beta \sum_{n=1}^{N-1} \|e_{p}^{n+1}\| 
\leq C\Big(h^{k}+h^{s+1}+\Delta t ^2\Big) &\text{for Option A},\\
&\Delta t \beta \sum_{n=1}^{N-1} \|I[e_{p}^{n+1}]\| 
\leq C\Big(h^{k}+h^{s+1}+\Delta t ^2\Big) &\text{for Option B}.
\end{aligned}
\end{equation}
\end{theorem}%
\begin{proof}Again, we only prove this for Option A since the other case requires only slight modification.
Using the equations (\ref{subtract_first_result}) and (\ref{deco_velo_error}) yields
\begin{equation}\label{pressure_subtract}
\begin{aligned}
&\left( \frac{D[\phi(t^{n+1})]}{\Delta t}, v_{h}\right) 
=\left( \frac{D[\eta(t^{n+1})]}{\Delta t}, v_{h}\right)
-b\left( I[{e}^{n+1}_{u}],I[u(t^{n+1})], v_{h}\right)  \\
&-b\left( I[{u}_h^{n+1}],I[{e}^{n+1}_{u}], v_{h}\right)  
-\nu\left( \nabla I[{e}^{n+1}_{u}], \nabla  v_{h}\right) \\
&+\left( {p}(t^{n+1})-{\lambda}^{n+1}_h, \nabla \cdot v_h \right)
+\tau^{n+1}(u,p;v_h) \quad \forall  v_{h}  \in {V}_{h}.\\
\end{aligned}
\end{equation}
We bound the six individual terms on the right hand side of (\ref{pressure_subtract}), term by term as follows:
\begin{equation}\label{p_err_1st}
\begin{aligned}
&\left( \frac{D[\eta(t^{n+1})]}{\Delta t}, v_{h}\right)
\leq C \Delta t^{-\frac{1}{2}} \|\eta_t\|_{L^2(t^{n-1},t^{n+1};L^2(\Omega))}\|\nabla v_h\|,\\
\end{aligned}
\end{equation}
\begin{equation}
\begin{aligned}
&-b\left( I[{e}^{n+1}_{u}],I[u(t^{n+1})], v_{h}\right) 
\leq C \| \nabla I[{e}^{n+1}_{u}]\| \|\nabla I[u(t^{n+1})]\| \|\nabla v_h\|,\\
\end{aligned}
\end{equation}
\begin{equation}
\begin{aligned}
&-b\left( I[{u}_h^{n+1}],I[{e}^{n+1}_{u}], v_{h}\right) 
 \leq C \|\nabla (I[{u}_h^{n+1}])\| \|\nabla I[{e}^{n+1}_{u}]\| \|\nabla  v_{h}\|,\\
\end{aligned}
\end{equation}
\begin{equation}
\begin{aligned}
&-\nu\left( \nabla I[{e}^{n+1}_{u}], \nabla  v_{h}\right)
\leq \nu \|\nabla I[{e}^{n+1}_{u}]\| \|\nabla  v_{h}\|,\\
\end{aligned}
\end{equation}
\begin{equation}
\begin{aligned}
&\left( {p}(t^{n+1})-{\lambda}^{n+1}_h, \nabla \cdot v_h \right)
 \leq C \|{p}(t^{n+1})-{\lambda}^{n+1}_h\| \|\nabla v_h\|,\\
\end{aligned}
\end{equation}
\begin{equation}\label{p_err_last}
\begin{aligned}
&\tau^{n+1}(u,p;v_h) 
 \leq C \Delta t^{\frac{3}{2}} \Big( \|u_{ttt}\|_{L^2(t^{n-1},t^{n+1};L^2(\Omega))} 
+ \|\nabla u_{tt}\|_{L^2(t^{n-1},t^{n+1};L^2(\Omega))}\\
&
+\|\nabla u\|^2_{L^4(t^{n-1},t^{n+1};L^2(\Omega))} + \| \nabla u_{tt}\|^2_{L^4(t^{n-1},t^{n+1};L^2(\Omega))}\Big)\|\nabla v_h\|.
\end{aligned}
\end{equation}
Considering equation ($\ref{pressure_subtract}$) and Lemma $\ref{dual_inequality}$ , using equations (\ref{p_err_1st})-(\ref{p_err_last}), dividing both sides by $\|\nabla v_h\|$ and taking a supremum over $V_h$ gives
\begin{equation} \label{rhd_estimates}
\begin{aligned}
&\bigg\|\frac{D[\phi(t^{n+1})]}{\Delta t}\bigg\|_{X_h^*}
\leq C \Big[ \Delta t^{-\frac{1}{2}} \|\eta_t\|_{L^2(t^{n},t^{n+1};L^2(\Omega))}\| \\
&+  \| \nabla I[{e}^{n+1}_{u}]\| (\|\nabla I[u(t^{n+1})]\|
+\|\nabla (I[{u}_h^{n+1}])\|
+1)\\
&+\|{p}(t^{n+1})-{\lambda}^{n+1}_h\|
+\Delta t^{\frac{3}{2}} \Big( \|u_{ttt}\|_{L^2(t^{n-1},t^{n+1};L^2(\Omega))} 
+ \|\nabla u_{tt}\|_{L^2(t^{n-1},t^{n+1};L^2(\Omega))}
\\
&+\|\nabla u\|^2_{L^4(t^{n-1},t^{n+1};L^2(\Omega))} + \| \nabla u_{tt}\|^2_{L^4(t^{n-1},t^{n+1};L^2(\Omega))}\Big)\Big].
\end{aligned}
\end{equation}
Separating the pressure error term $e_p^{n+1}=(p(t^{n+1})-\lambda^{n+1}_h) - (p^{n+1}_h-\lambda^{n+1}_h)$ and rearranging implies
\begin{gather*}
\Big(p_h^{n+1}-{\lambda}^{n+1}_h, \nabla \cdot v_h \Big) 
=\left( \frac{D[\eta(t^{n+1})]}{\Delta t}, v_{h}\right)
-\left( \frac{D[\phi(t^{n+1})]}{\Delta t}, v_{h}\right) \\
+\nu\left( \nabla I[{e}^{n+1}_{u}], \nabla  v_{h}\right) 
-\left( e_{p}^{n+1}, \nabla \cdot v_h \right)
-\left( {p}(t^{n+1})-{\lambda}^{n+1}_h, v_{h}\right) +\tau^{n+1}(u,p;v_h)   \quad \forall  v_{h}  \in {X}_{h}.
\end{gather*}
Consider the estimates in (\ref{p_err_1st})-(\ref{rhd_estimates}). Divide by $\|\nabla v_h\|$, take supremum over $v_h \in X_h$ and use discrete inf-sup condition to obtain,
\begin{equation} \label{no_square}
\begin{aligned}
&\beta\|p_h^{n+1}-{\lambda}^{n+1}_h\|
\leq  C \Big[ \Delta t^{-\frac{1}{2}} \|\eta_t\|_{L^2(t^n,t^{n+1};L^2(\Omega))} \\
&+  \| \nabla I[{e}^{n+1}_{u}]\| \Big(\|\nabla I[u(t^{n+1})]\|
+\|\nabla (I[{u}_h^{n+1}])\|
+1\Big)\\
&+\|{p}(t^{n+1})-{\lambda}^{n+1}_h\|
+\Delta t^{\frac{3}{2}} \Big( \|u_{ttt}\|_{L^2(t^{n-1},t^{n+1};L^2(\Omega))} + \|\nabla u_{tt}\|_{L^2(t^{n-1},t^{n+1};L^2(\Omega))}
\\
&+\|\nabla u\|^2_{L^4(t^{n-1},t^{n+1};L^2(\Omega))} + \| \nabla u_{tt}\|^2_{L^4(t^{n-1},t^{n+1};L^2(\Omega))}\Big)\Big].
\end{aligned}
\end{equation}
We multiply by $\Delta t$, sum from $n=1$ to $n=N-1$ and apply triangle inequality. This yields 
\begin{equation}
\begin{aligned}
&\beta\Delta t \sum_{n=1}^{N-1}\|e_{p}^{n+1}\|
\leq  C \Big[ \Delta t^{-\frac{1}{2}} \|\eta_t\|_{L^2(0,T;L^2(\Omega))}\\
&+ \||p(t^{n+1})-\lambda^{n+1}_h\||_{2,0}
+  \|| \nabla I[{e}^{n+1}_{u}]\||_{2,0} \\
&+ \Delta t ^{\frac{5}{2}} \Big( \| u_{ttt}\|_{2,0}
+ \|\nabla u_{tt}\|_{2,0}
+\||\nabla u \||^2_{4,0} + \| \nabla u_{tt}\|^2_{4,0}\Big)\Big].
\end{aligned}
\end{equation}
Results from the equations $(\ref{bound_after_grwonwall_second_term})$ and $(\ref{bound_after_grwonwall_6th_term})$ give the bounds for the first two terms. Using error estimates of the velocity on the third term and taking infimum over $X_h$ and $Q_h$ yield the result.
\end{proof}

\section{Numerical tests}

\label{sec:tests} We verify second order convergence for the new method
through an exact solution in Section \ref{sec:taylor}. Visualizations of the
flow and benchmark quantities gives additional
support to the increased accuracy of the new method in Section \ref%
{sec:cylinder}. The tests used $P_{2}/P_{1}$ and $P_{3}/P_{2}$ elements. All
computations were performed with FEniCS \cite{fenics}.

\subsection{Taylor-Green vortex \label{sec:taylor}}

We apply the backward Euler and the backward Euler plus filter for the 2D
Taylor-Green vortex. This test problem is historically used to assess accuracy
and convergence rates in CFD \cite{chorin1967}. The exact solution is given
by 
\begin{equation*}
u=e^{-2\nu t}(\cos x\sin y,-\sin x\cos y)\text{ and }p=-\frac{1}{4}e^{-4\nu
t}(\cos 2x+\cos 2y).
\end{equation*}%
To test time accuracy, we solve using $P_{3}/P_{2}$ elements on a uniform
mesh of $250\times 250$ squares divided into 2 triangle per square. We take
a series of time steps for which the total error is expected to be
dominated by the temporal error. Since the true solution decays
exponentially, we tabulate and display relative errors. Fig. 6.1 displays
the relative errors for backward Euler, backward Euler plus filtering only
the velocity (Algorithm \ref{alg:themethod}A), and backward Euler plus filtering both the velocity and
pressure (Algorithm \ref{alg:themethod}B). Filtering the pressure does not affect the velocity solution, so
the velocity error plot only shows two lines. The velocity error is $%
\mathcal{O}(\Delta t^{2})$, as predicted, and significantly smaller than the
backward Euler error. Thus, adding the filter step (1.3) reduces the
velocity error substantially, Figure 8.1, at negligible cost, Figure 1.1.
The pressure error is $\mathcal{O}(\Delta t^{2})$ when either both $u$ and $p
$ are filtered, or only $u$ is filtered, which is consistent with our
theoretical analysis. Filtering only $u$ has smaller pressure error since
the pressure filter introduces an extra consistency error term, see \eqref{eqn:consistency_defnb}.

\begin{figure}[tbp]
\centering
\begin{subfigure}{0.49\linewidth}
   \centering
   \includegraphics[width = 1\linewidth]{velocityError.pdf}
\end{subfigure}
\begin{subfigure}{0.49\linewidth}
   \centering
   \includegraphics[width = 1\linewidth]{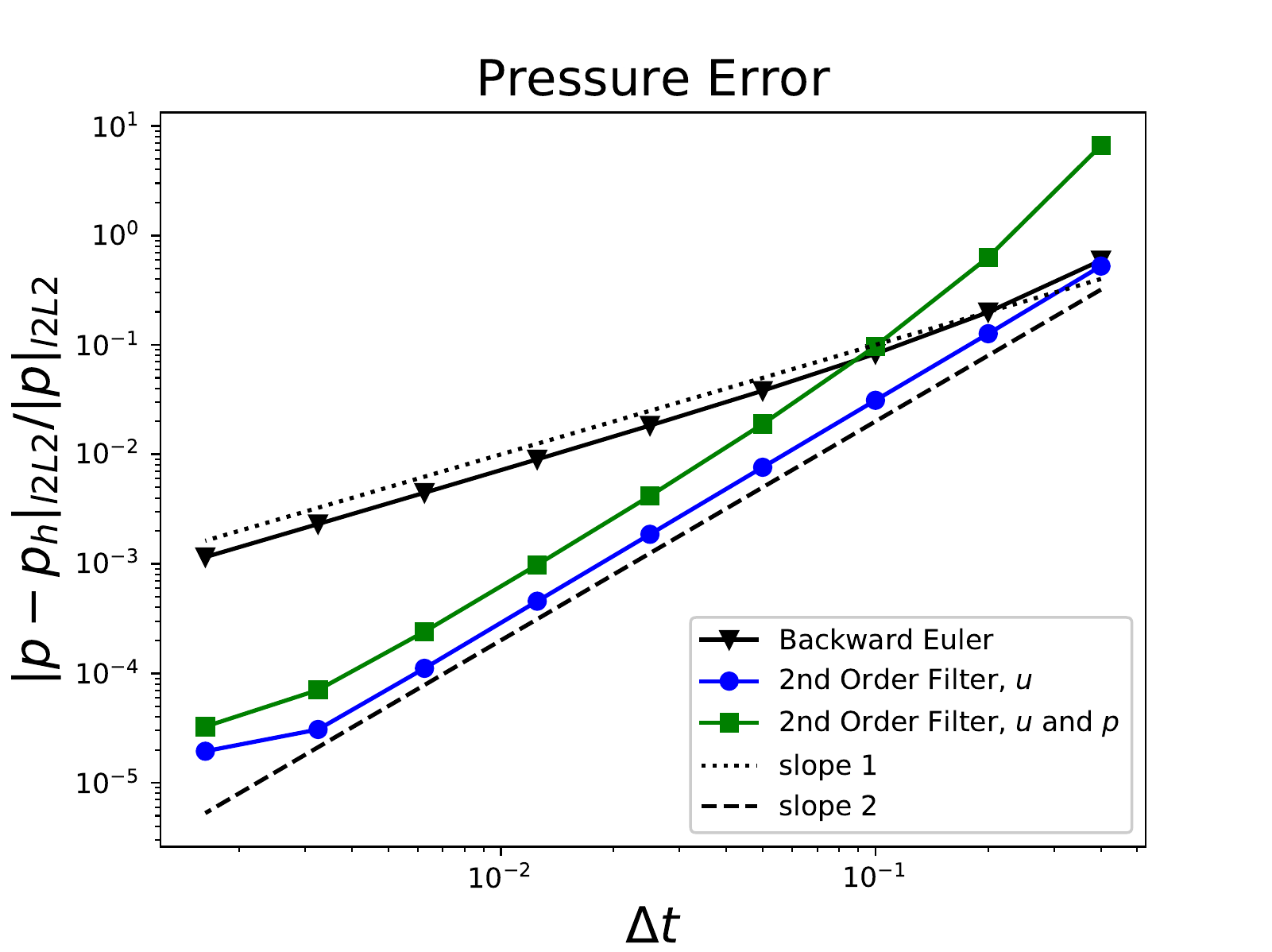}
\end{subfigure}
\caption{Convergence rates for the filtered quantities are second order as
predicted. Filtering only the velocity produces the best pressure.\label{fig:convergence}}
\end{figure}
\subsection{Adaptive Test\label{sec:test_adaptive}}
We test the time/order adaptive algorithm on a problem that showcases the superiority of the VSVO method over the constant stepsize, constant order method.

The Taylor-Green problem can be modified by replacing $F$ with any differentiable function of $t$. With velocity and pressure defined as before, the required body force is
\begin{equation}\notag
f(x,y,t) = (2 \nu F(t) + F'(t))\langle \cos x \sin y, -\cos y \sin x \rangle.
\end{equation}
For $F(t)$, we construct a sharp transition function between 0 and 1. First, let 
\begin{equation}\notag
    g(t) =
    \begin{cases}
      0        & \text{if } t \leq 0\\
      \exp\left(-\frac{1}{(10t)^{10}}\right) & \text{if }  t> 0 \\
    \end{cases}
  \end{equation}
This is a differentiable function, and $g(5) \equiv 1$ in double precision. Therefore, a differentiable (up to machine precision) function can be constructed with shifts and reflections of this function. This creates sections of flatness, and sections that rapidly change which require adaptivity to resolve efficiently. See Fig. \ref{fig:cutoff} for the evolution of $\|u\|$ with time. All tests were initialized at rest spaced at a constant interval of $k = 0.1$, 100 nodes per side of the square using $P_2/P_1$ elements, and with final time of 45.

Figure \ref{fig:cutoff} compares two numerical solutions. One is from Algorithm \ref{alg:themethod} (second order - nonadaptive), and the other is from Algorithm \ref{alg:vsvo_themethod} (VSVO-12). With $TOL=10^{-3}$, the VSVO-12 method takes 342 steps, which comprises 254 accepted steps, and 88 rejected steps. The constant stepsize method which took 535 steps does not accurately capture the energetic jumps.

Figure \ref{fig:work} shows the relative $l^2L^2$ velocity errors versus steps taken of VSVO-12 for seven different $TOL$s, starting at $10^{-1}$, and dividing by ten down to $10^{-7}$. This is compared with nonadaptive method (which has no rejected steps) sampled at several stepsizes. Both methods show second order convergence, but for smaller tolerances, VSVO-12 performs about $10^3$ better than the nonadaptive method for the same amount of work. 
\begin{figure}[!ht]

   \includegraphics[width = 1\linewidth]{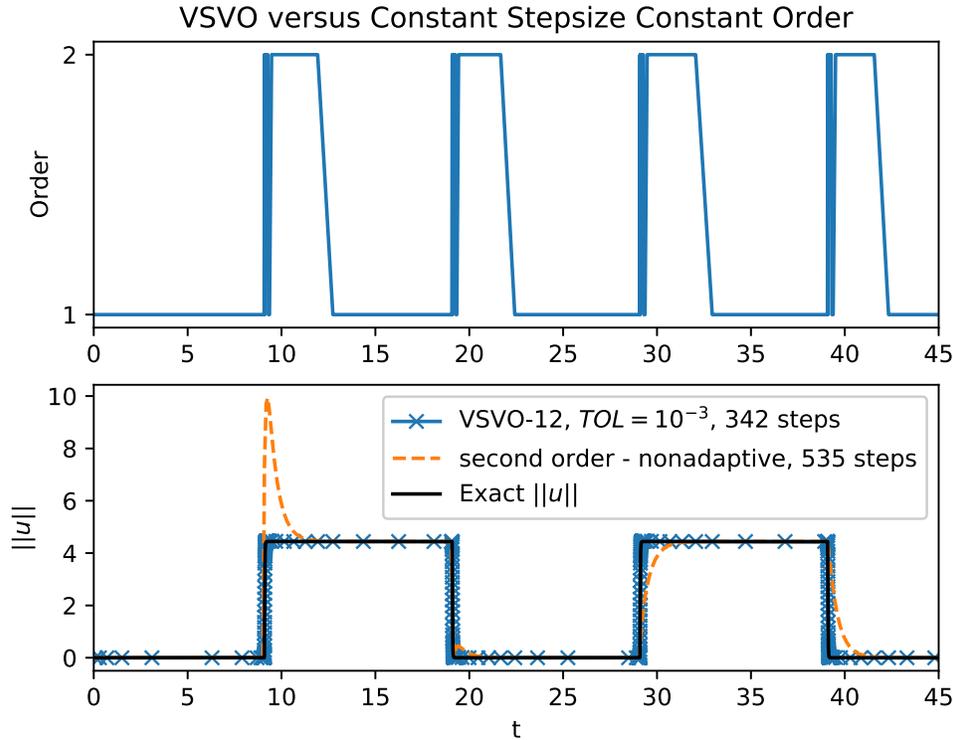}

\caption{The nonadaptive second order method results in large overshoots and undershoots while requiring more work than the adaptive method.\label{fig:cutoff}}
\end{figure}

\begin{figure}[!ht]

   \includegraphics[width = 1\linewidth]{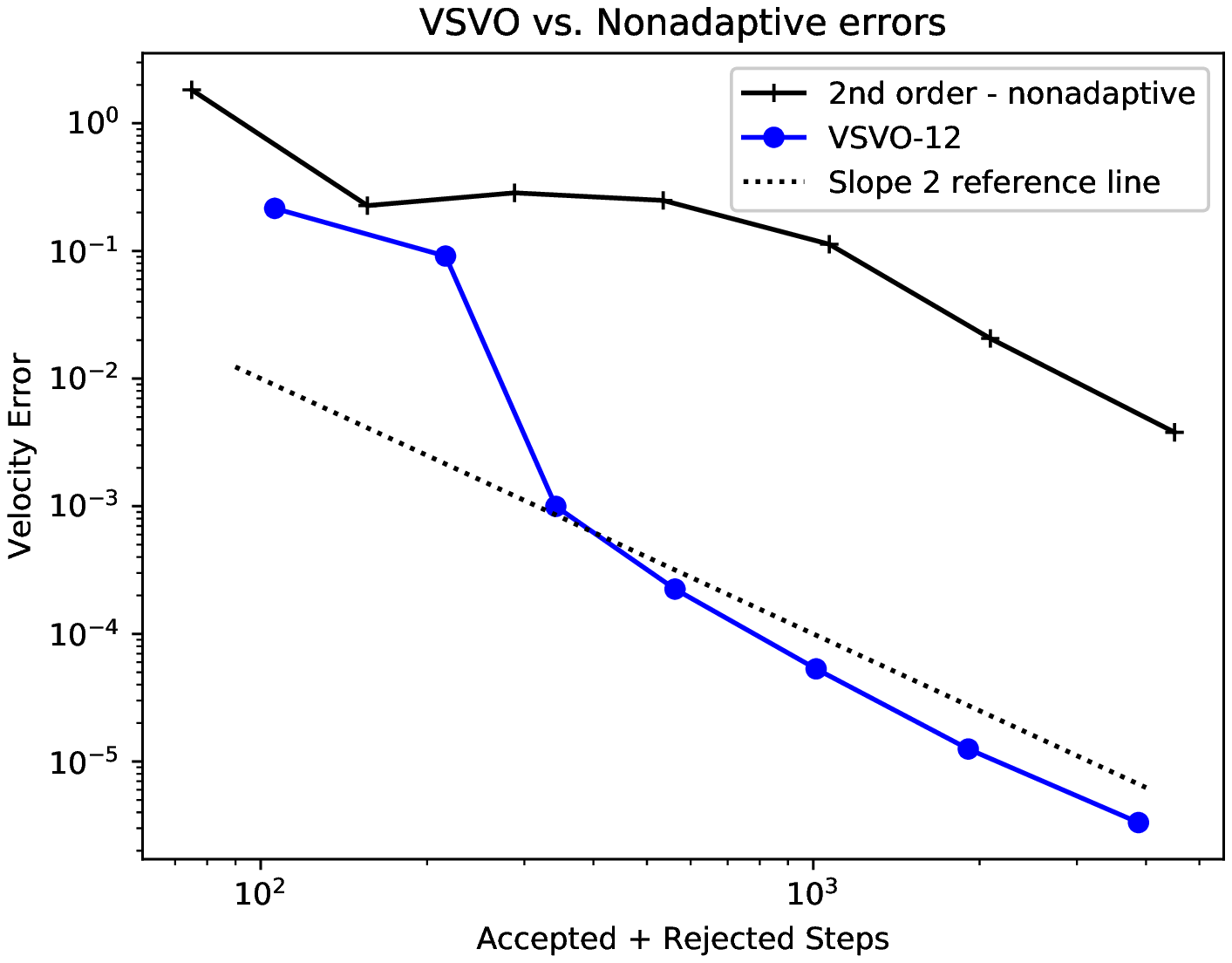}

\caption{The VSVO-12 method performs three orders of magnitude better for the same amount of work compared to the nonadaptive 2nd order method for the test problem in Section \ref{sec:test_adaptive}. Each circle represents a different tolerance from $TOL=10^{-1}$ to $10^{-7}$.\label{fig:work}}
\end{figure}

\subsection{Flow around a cylinder\label{sec:cylinder}}

We now use the benchmark problem of flow around a cylinder, originally
proposed in \cite{schaefer}, to test the improvement obtained using filters
on flow quantities (drag, lift, and pressure drop) using values obtained via
a DNS in \cite{FLD:FLD679} as a reference. This problem has also been used as a
benchmark in \cite{Rang2005},\cite{LIU20103428},\cite{besier2012},\cite%
{CHARNYI2017289} and others. Let $\nu = 10^{-3}$, $f \equiv 0$, $T_{final}=8$%
, and 
\begin{equation}
\Omega = \{(x,y) \,|\, 0 < x < 2.2, \hspace{2mm} 0 < y < 0.41\text{ and }
(x-0.2)^2 + (y - 0.2)^2 > 0.05^2\},  \notag
\end{equation}
i.e., a channel with a cylindrical cutout. A parabolic velocity of $%
u=0.41^{-2}\sin(\pi t/8)(6y(0.41-y),0)$ is prescribed at the left and right
boundaries. We used a spatial discretization with 479026 degrees of freedom
with 1000 vertices on the boundary of the cylinder. The mesh used $P2/P1$
elements, and was obtained by adaptive refinement from solving the steady
solution with $u=0.41^{-2}(6y(0.41-y),0)$ as inflow and outflow boundary conditions.

The correct behavior for this problem is that vortices shed off the cylinder
as the inlet and outlet velocities increase. Fig. \ref{fig:conv-by-pic}
shows snapshots of the flow at $t=6$ for five successively halved $\Delta t$%
's. The Backward Euler approximation shows no vortex shedding for $\Delta
t=0.04,0.02$, and $0.01$. The filtered method of Algorithm \ref{alg:themethod}
shows the qualitatively correct behavior from $\Delta t=0.02$ on. Clearly,
higher order and less dissipative methods are necessary to see dynamics for
modestly large $\Delta t$.

It was demonstrated in \cite{FLD:FLD679} that the backward Euler time
discretization greatly under predicts lift except for very small step sizes.
Fig. \ref{fig:lift_drag} demonstrates that the time filter in Algorithm \ref%
{alg:themethod} corrects both the amplitude and phase error in the backward
Euler approximation. Other quantities that were compared to reference values
were the maximum drag $c_{d,\max}$, the time of max drag $t(c_{d,\max})$,
time of maximum lift $t(c_{l,\max})$, and pressure drop across the cylinder
at $t=8$ are shown in Table \ref{tab:cutoff}.

The choice of whether or not to filter the pressure does not affect the
velocity solution, the snapshots shown Figure \ref{fig:conv-by-pic} are the
same for both choices. Table \ref{tab:cutoff} shows that filtering $u$
greatly improves the calculated flow quantities whether or not $p$ is
filtered.

\begin{figure}[!ht]
\centering
\begin{subfigure}{1\linewidth}
   \centering
   \includegraphics[width = 1\linewidth]{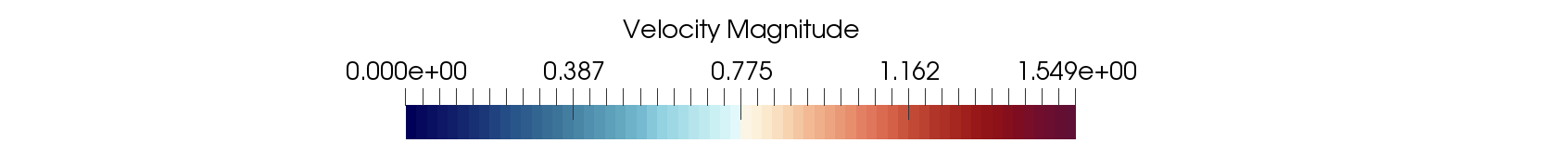}
\end{subfigure}
\par
\begin{subfigure}{0.49\linewidth}
   \centering
   \includegraphics[width = 1\linewidth]{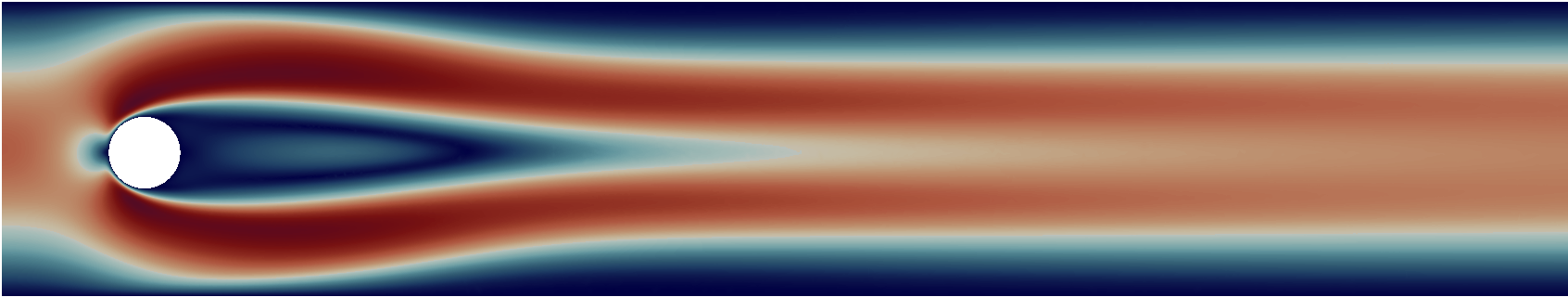}
\end{subfigure}
\begin{subfigure}{0.49\linewidth}
   \centering
   \includegraphics[width = 1\linewidth]{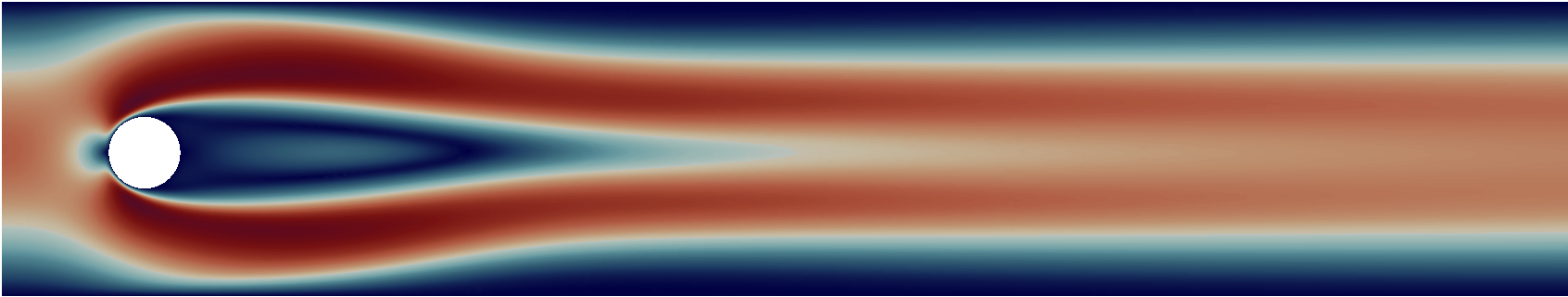}
\end{subfigure}
\begin{subfigure}{0.49\linewidth}
   \centering
   \includegraphics[width = 1\linewidth]{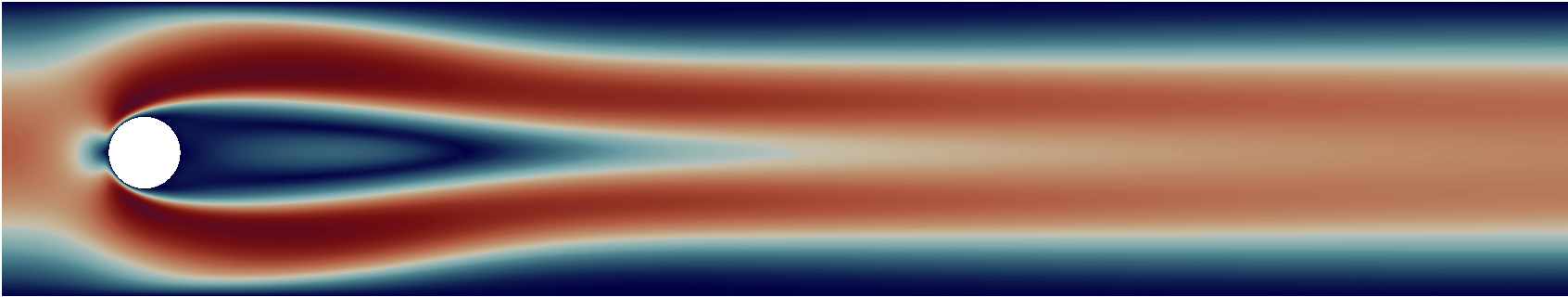}

\end{subfigure}
\begin{subfigure}{0.49\linewidth}
   \centering
   \includegraphics[width = 1\linewidth]{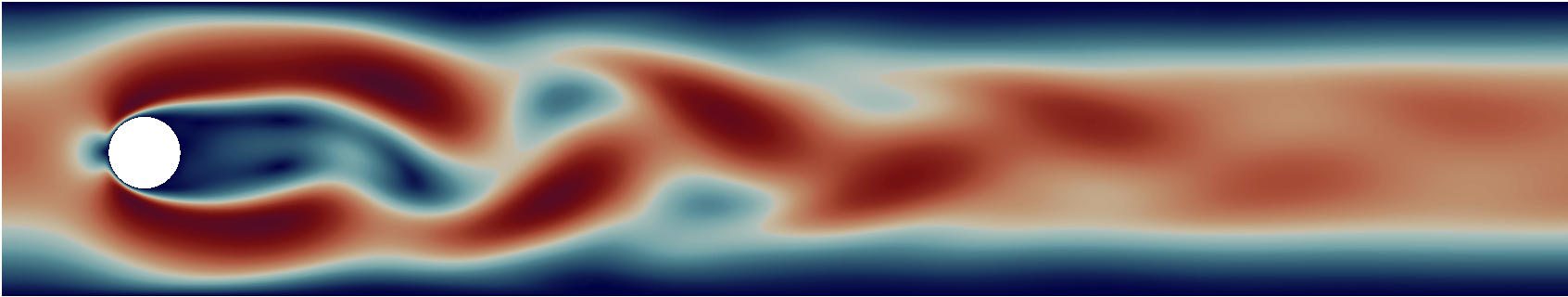}

\end{subfigure}
\begin{subfigure}{0.49\linewidth}
   \centering
   \includegraphics[width = 1\linewidth]{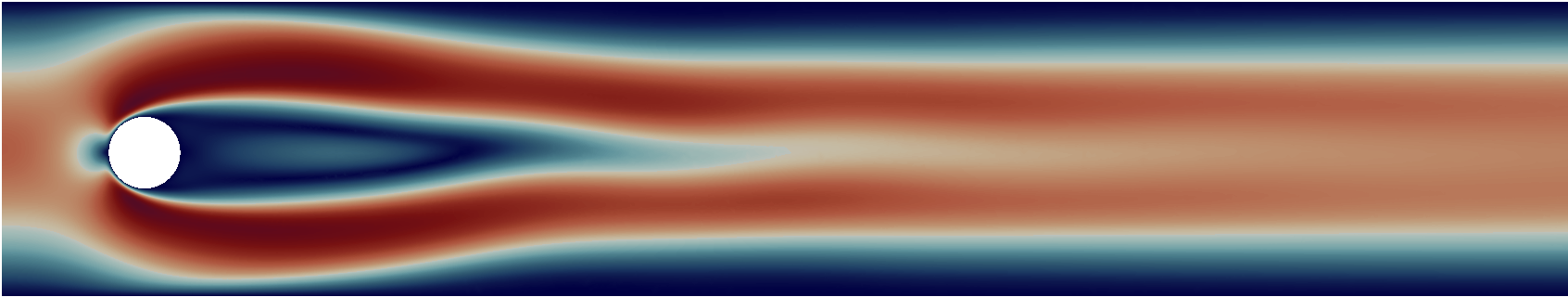}

\end{subfigure}
\begin{subfigure}{0.49\linewidth}
   \centering
   \includegraphics[width = 1\linewidth]{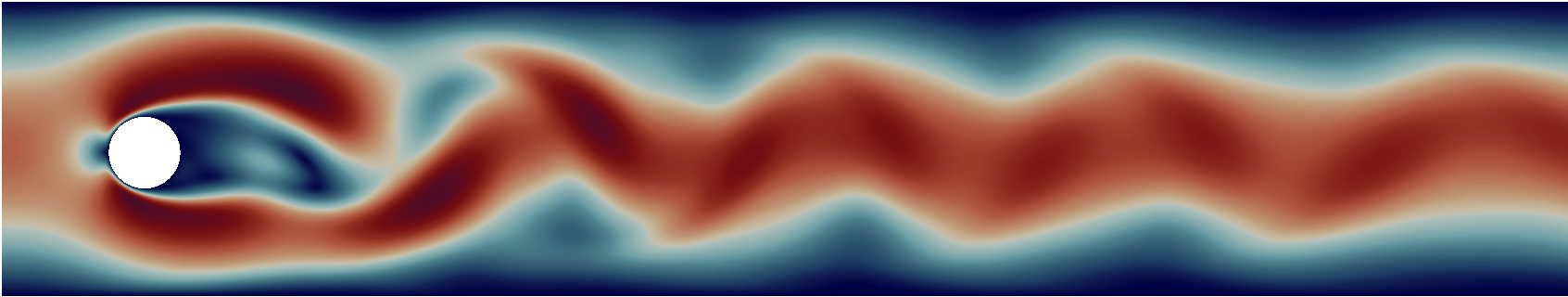}

\end{subfigure}
\begin{subfigure}{0.49\linewidth}
   \centering
   \includegraphics[width = 1\linewidth]{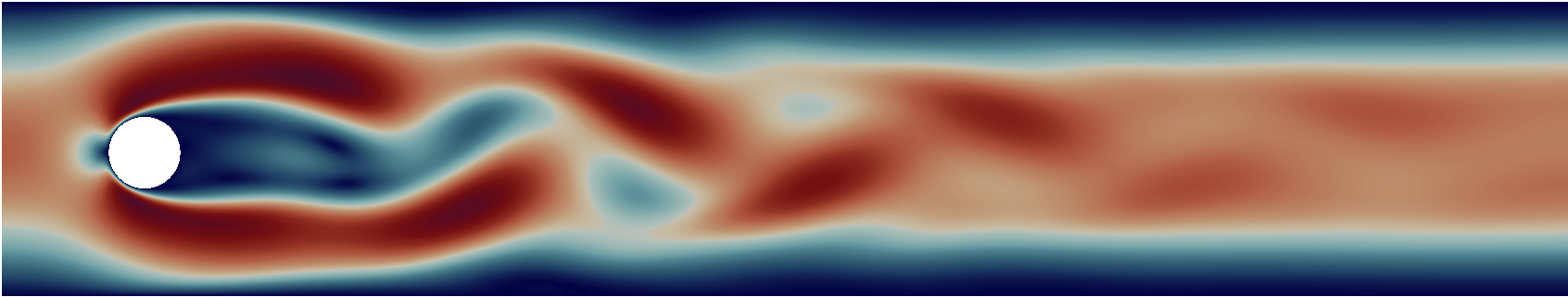}

\end{subfigure}
\begin{subfigure}{0.49\linewidth}
   \centering
   \includegraphics[width = 1\linewidth]{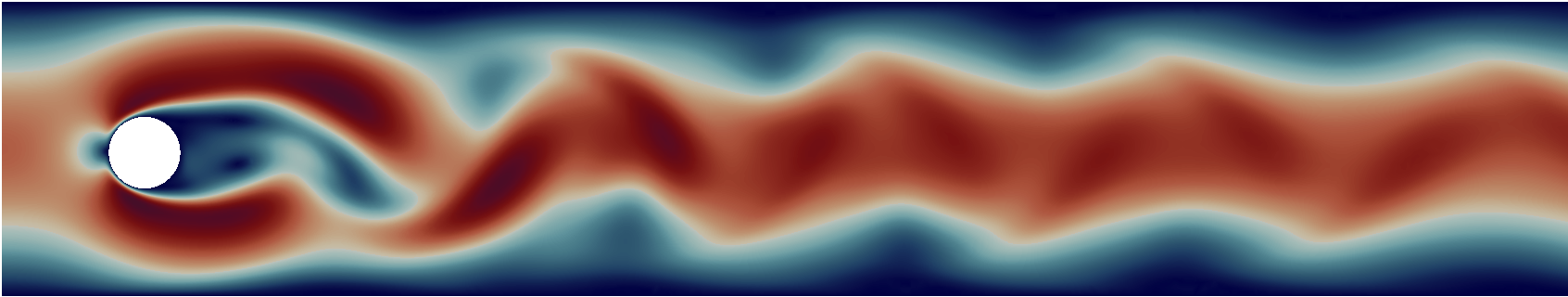}

\end{subfigure}
\begin{subfigure}{0.49\linewidth}
   \centering
   \includegraphics[width = 1\linewidth]{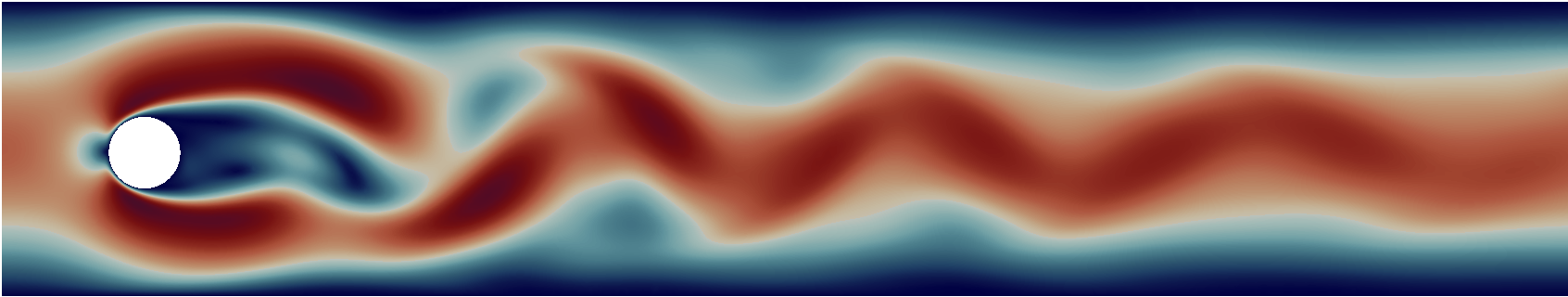}
   \caption{Backward Euler}

\end{subfigure}
\begin{subfigure}{0.49\linewidth}
   \centering
   \includegraphics[width = 1\linewidth]{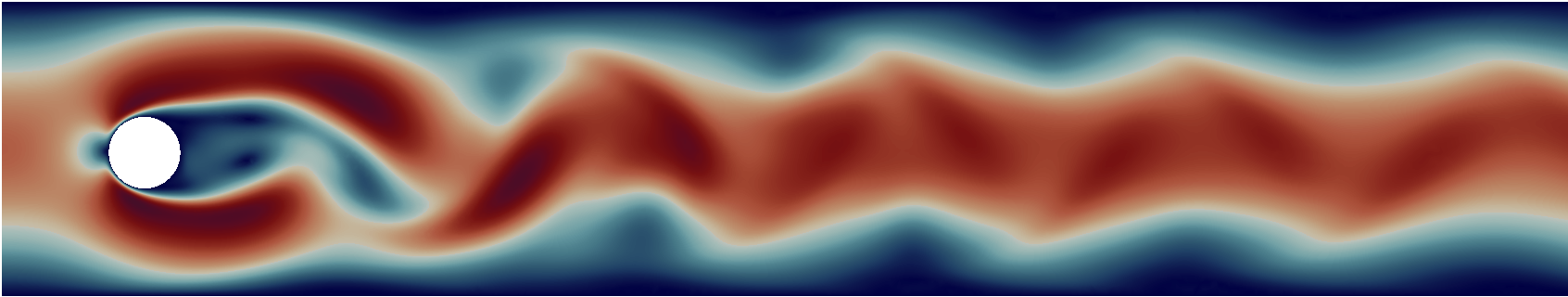}
   \caption{Backward Euler Plus Filter}

\end{subfigure}
\caption{Flow snapshots at $t=6$ with $\Delta t = 0.04$ (top), and $\Delta t$
halving until $\Delta t = 0.0025$ (bottom). Backward Euler (left) destroys
energy and suppresses oscillations, meaning that it can predict nearly
steady state solutions when a time dependent one exists. The time filter
(right) corrects this. }
\label{fig:conv-by-pic}
\end{figure}
\begin{figure}[!ht]
\centering
\begin{subfigure}{0.49\linewidth}
   \centering
   \includegraphics[width = 1\linewidth]{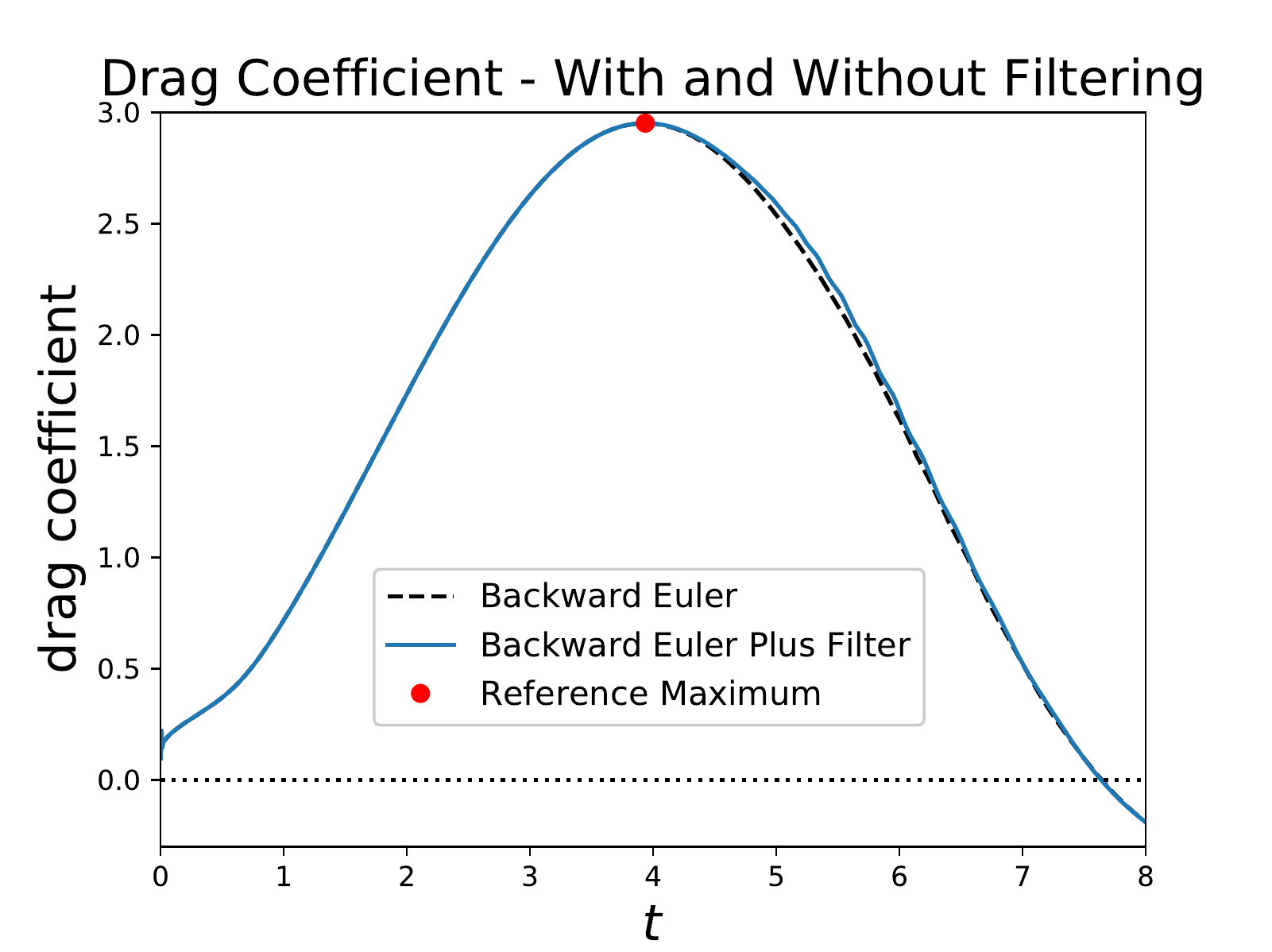}
\end{subfigure}
\begin{subfigure}{0.49\linewidth}
   \centering
   \includegraphics[width = 1\linewidth]{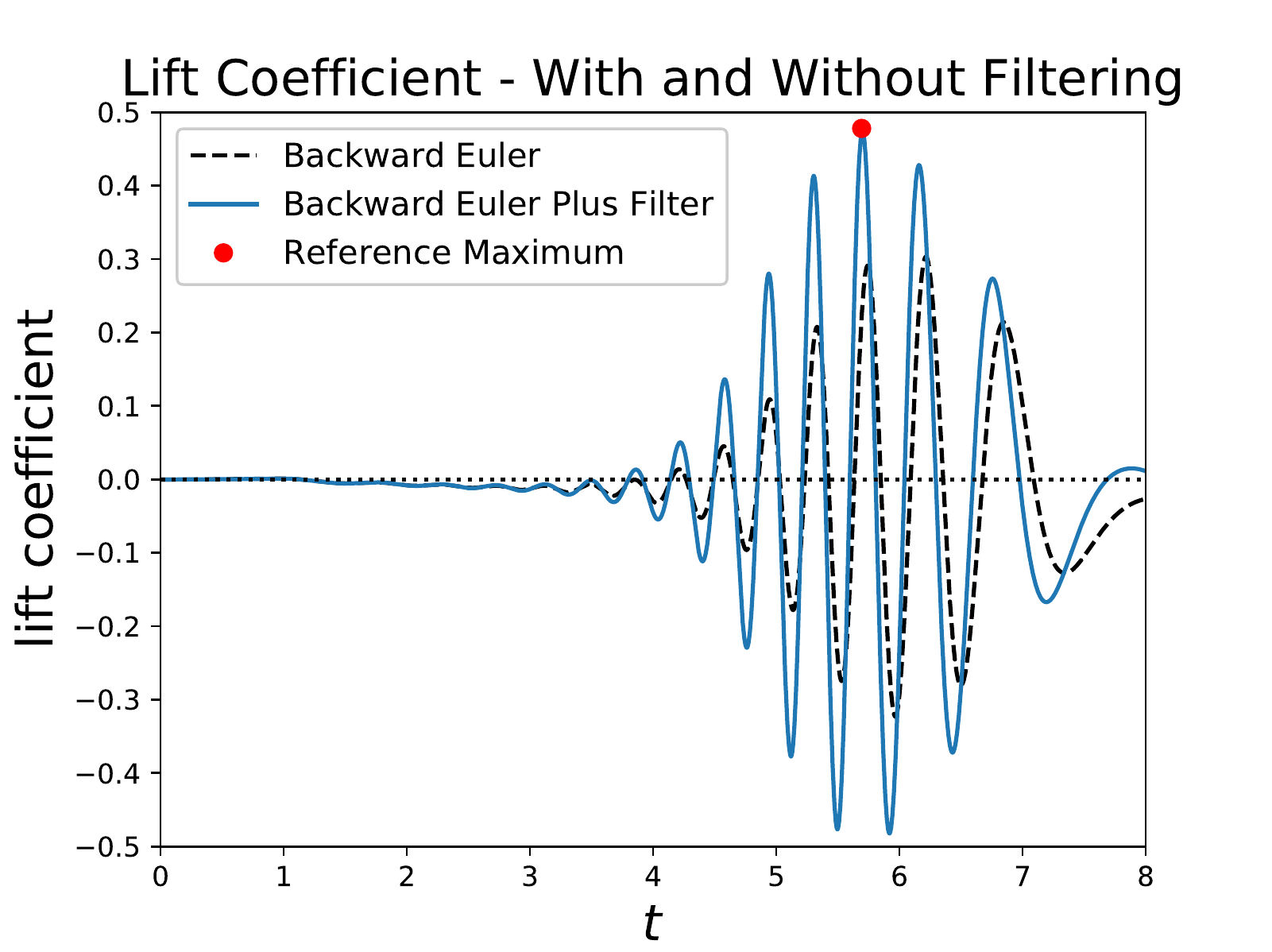}
\end{subfigure}
\par
\begin{subfigure}{0.49\linewidth}
   \centering
   \includegraphics[width = 1\linewidth]{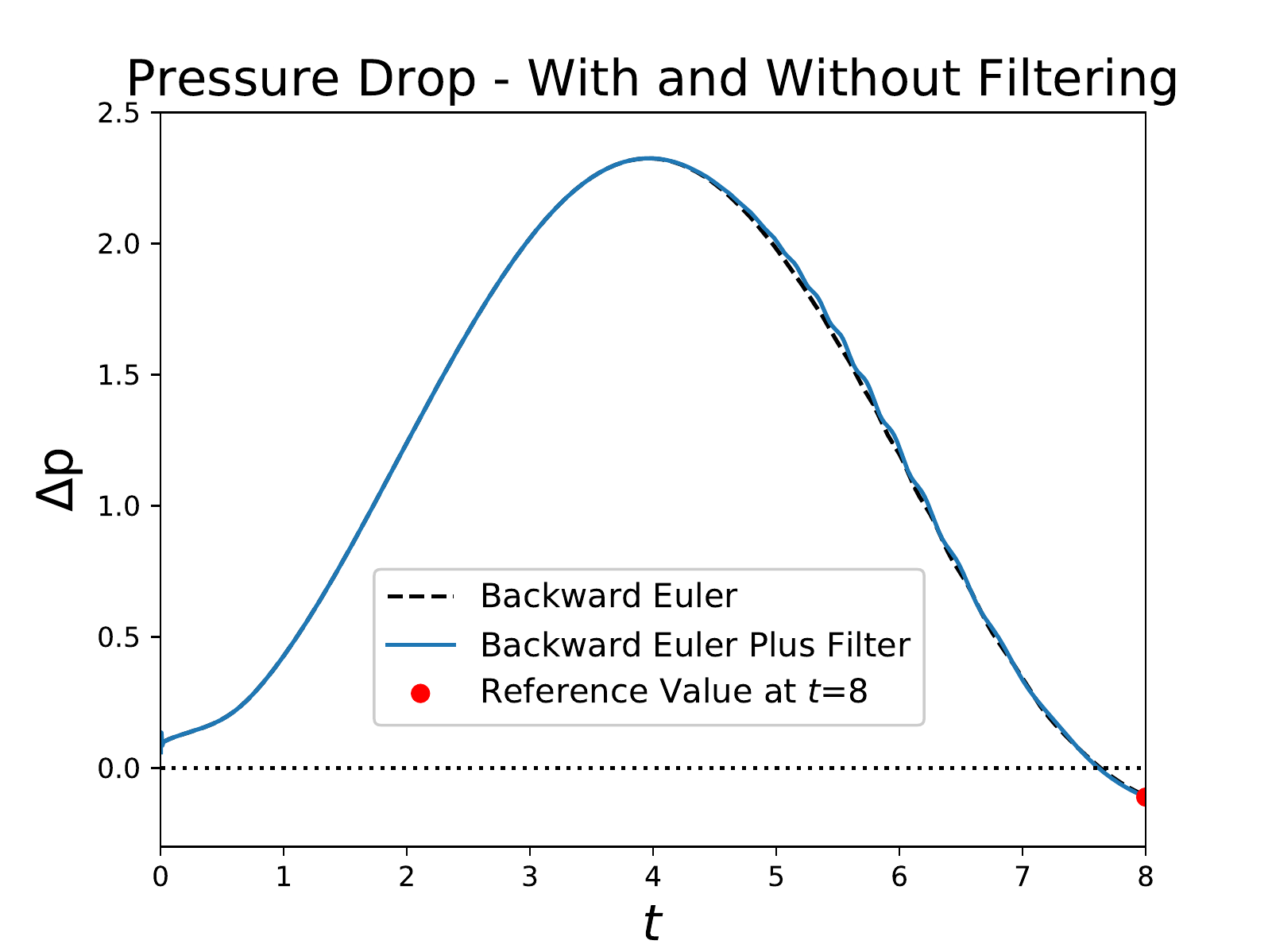}
\end{subfigure}
\caption{Lift of the Backward Euler solution and the filtered solution for $%
\Delta t = 0.0025$. The filtered solution correctly predicts both the time
and magnitude of the maximum lift. }
\label{fig:lift_drag}
\end{figure}

\begin{table}[tbp]
\caption{Lift, drag, and pressure drop for cylinder problem}
\label{tab:cutoff}\centering
Backward Euler\newline
\begin{tabular}{p{1.5cm}p{1.5cm}p{2cm}p{1.5cm}p{2cm}p{2cm}}
\hline
$\Delta t$ & $t(c_{d,\max})$ & $c_{d,\max}$ & $t(c_{l,\max})$ & $c_{l,\max}$
& $\Delta p(8)$ \\ \hline
0.04 & 3.92 & 2.95112558 & 0.88 & 0.00113655 & -0.12675521 \\ 
0.02 & 3.94 & 2.95064522 & 0.92 & 0.00117592 & -0.12647232 \\ 
0.01 & 3.93 & 2.95041574 & 7.17 & 0.02489640 & -0.12433915 \\ 
0.005 & 3.93 & 2.95031983 & 6.28 & 0.17588270 & -0.10051423 \\ 
0.0025 & 3.9325 & 2.95038901 & 6.215 & 0.30323034 & -0.10699361 \\ 
&  &  &  &  & 
\end{tabular}%
\newline
\vspace{3mm} Backward Euler Plus Filter\newline
\begin{tabular}{p{1.5cm}p{1.5cm}p{2cm}p{1.5cm}p{2cm}p{2cm}}
\hline
0.04 & 3.92 & 2.95021463 & 7.56 & 0.00438111 & -0.12628328 \\ 
0.02 & 3.94 & 2.95026781 & 6.14 & 0.20559211 & -0.11146505 \\ 
0.01 & 3.93 & 2.95060684 & 5.81 & 0.40244197 & -0.09943203 \\ 
0.005 & 3.935 & 2.95082513 & 5.72 & 0.46074771 & -0.11111586 \\ 
0.0025 & 3.935 & 2.95089028 & 5.7 & 0.47414096 & -0.11193754 \\ \hline
\end{tabular}
\vspace{3mm}
\par
Backward Euler Plus Filter $u$ and $p$\newline
\begin{tabular}{p{1.5cm}p{1.5cm}p{2cm}p{1.5cm}p{2cm}p{2cm}}
\hline
0.04 & 3.92 & 2.95073993 & 7.52 & 0.00439864 & -0.12642684 \\ 
0.02 & 3.94 & 2.95039973 & 6.14 & 0.21101313 & -0.11153593 \\ 
0.01 & 3.93 & 2.95063962 & 5.81 & 0.40624697 & -0.09945143 \\ 
0.005 & 3.935 & 2.95083296 & 5.72 & 0.46192306 & -0.11112049 \\ 
0.0025 & 3.935 & 2.95089220 & 5.7 & 0.47444753 & -0.11193859 \\ \hline
\end{tabular}
\vspace{3mm}
\par
Reference Values\newline
\begin{tabular}{p{1.5cm}p{1.5cm}p{2cm}p{1.5cm}p{2cm}p{2cm}}
\hline
--- & 3.93625 & 2.950921575 & 5.693125 & 0.47795 & $-$0.1116 \\ \hline
\end{tabular}%
\end{table}
\section{Conclusion}

Accurate and stable time discretization is important for obtaining correct
flow predictions. The backward Euler time discretization is a stable but
inaccurate method. We have shown that for minimum extra programming effort,
computational complexity, and storage, second order accuracy and
unconditional stability can be obtained by adding a time filter. Due to the
embedded and modular structure of the algorithm, both adaptive time-step and
adaptive order are easily implemented in a code based on a backward Euler
time discretization. Extension of the method and analysis to yet higher
order time discretization is important as is exploring the effect of time
filters on other methods possible for Step 1 of Algorithm 1.1. Analysis of
the effect of time filters with moving and time dependent boundary
conditions would also be a significant extension.

\noindent
\textbf{Acknowledgment}

The research herein was partially supported by NSF grants DMS 1522267, 1817542 and CBET 160910.

\bibliographystyle{abbrv}
\bibliography{references}

\appendix
\section{Velocity Error Analysis\label{app:velocity}}

\ 

\subsection{Proof of Lemma \protect\ref{consis_lemma}}

\ 
\begin{proof}
By Taylor's theorem with the integral remainder,
\begin{align*}\notag
D[u(t^{n+1})] - \Delta t u_t(t^{n+1})= \frac{3}{2}u(t^{n+1})- \Delta t u_t(t^{n+1})\\ -2\left(u(t^{n+1})-\Delta t u_t(t^{n+1}) + \frac{\Delta t^2}{2}u_{tt}(t^{n+1})) +\frac{1}{2} \int_{t^{n+1}}^{t^n}u_{ttt}(t)(t^n-t)^2dt \right)\\
+\frac{1}{2}\left(u(t^{n+1})-2\Delta t u_t(t^{n+1}) + 2\Delta t^2u_{tt}(t^{n+1})) +\frac{1}{2} \int_{t^{n+1}}^{t^{n-1}}u_{ttt}(t)(t^{n-1}-t)^2dt \right)\\
=-\int_{t^{n}}^{t^{n+1}}u_{ttt}(t^n-t)^2dt - \frac{1}{4} \int_{t^{n-1}}^{t^{n+1}}u_{ttt}(t^{n-1}-t)^2dt.
\end{align*}
These terms are first estimated by Cauchy-Schwarz.
\begin{equation}\notag
\left(\int_{t^{n}}^{t^{n+1}}u_{ttt}(t)(t^n-t)^2dt\right)^2 \leq \int_{t^{n}}^{t^{n+1}}u_{ttt}^2dt \int_{t^{n}}^{t^{n+1}}(t^n-t)^4dt = \frac{\Delta t^5}{5}\int_{t^{n}}^{t^{n+1}}u_{ttt}^2dt.
\end{equation}
\begin{equation}\notag
\frac{1}{16}\left(\int_{t^{n-1}}^{t^{n+1}}u_{ttt}(t)(t^{n-1}-t)^2dt\right)^2 \leq \frac{1}{16}\int_{t^{n-1}}^{t^{n+1}}u_{ttt}^2dt \int_{t^{n-1}}^{t^{n+1}}(t^{n-1}-t)^4dt = \frac{2\Delta t^5}{5}\int_{t^{n-1}}^{t^{n+1}}u_{ttt}^2dt.
\end{equation}
Thus,
\begin{equation}\notag
\left( {\frac{D[u(t^{n+1})]}{\Delta t}} - u_t(t^{n+1})\right)^2 \leq \frac{6}{5}\Delta t^3\int_{t^{n-1}}^{t^{n+1}}u_{ttt}^2dt.
\end{equation}
Integrating with respect to $x$ yields the first inequality.
Next,
\begin{align*}
I[u(t^{n+1})] - u(t^{n+1}) = \frac{1}{2}u(t^{n+1}) - u(t^{n}) + \frac{1}{2}u(t^{n-1})\\
= \int_{t^n}^{t^{n+1}} u_{tt}(t) (t^{n+1}-t)dt + \int_{t^n}^{t^{n-1}} u_{tt}(t) (t^{n-1}-t)dt.
\end{align*}
By similar steps,
\begin{equation}\notag
\left(\int_{t^{n}}^{t^{n+1}}u_{tt}(t)(t^n-t)dt\right)^2 \leq \frac{\Delta t^3}{3}\int_{t^{n}}^{t^{n+1}}u_{tt}^2dt.
\end{equation}
\begin{equation}\notag
\left(\int_{t^{n-1}}^{t^{n}}u_{tt}(t)(t^{n-1}-t)dt\right)^2 \leq \frac{\Delta t^3}{3}\int_{t^{n-1}}^{t^{n}}u_{tt}^2dt.
\end{equation}
Therefore,
\begin{equation}
\left(I[u(t^{n+1})] - u(t^{n+1})\right)^2 \leq \frac{4}{3}\Delta t^3 \int_{t^{n-1}}^{t^{n+1}} u_{tt}^2 dt.
\end{equation}
The last inequality can be proved using the same strategy.
\end{proof}

\subsection{Proof of Theorem \protect\ref{convergency}\label{sec:velocity_error}}

\ 
\begin{proof}We prove this for Option A. A parallel proof exists for Option B. At $t^{n+1}=(n+1)\Delta t$, the true solution of (\ref{NSE}) satisfies, 
\begin{equation}\label{true_solution}
\begin{aligned}
&\left( \frac{D[u(t^{n+1})]}{\Delta t}, v_{h}\right) 
+b\left( I[u(t^{n+1})],I[u(t^{n+1})], v_{h}\right)  \\
&+\nu\left( \nabla I[u(t^{n+1})], \nabla  v_{h}\right) 
-\left( {p}(t^{n+1}),\nabla \cdot  v_{h}\right) \\
&=\left( \mathbf{f}^{n+1}, v_{h}\right) + \tau^{n+1}(u,p;v_h) \quad \forall  v_{h}  \in {X}_{h}.\\
\end{aligned}
\end{equation}
Subtracting $(\ref{eqn:equiv_method_full})$ from $(\ref{true_solution})$ yields
\begin{equation}\label{subtract_first_result}
\begin{aligned}
&\left( \frac{D[{e}^{n+1}_{u}]}{\Delta t}, v_{h}\right) 
+b\left( I[{e}^{n+1}_{u}],I[u(t^{n+1})], v_{h}\right)  \\
&+b\left( I[{u}_h^{n+1}],I[{e}^{n+1}_{u}], v_{h}\right)  
+\nu\left( \nabla I[{e}^{n+1}_{u}], \nabla  v_{h}\right) \\
&-\left( e_{p}^{n+1}, \nabla \cdot v_h \right)
=\tau^{n+1}(u,p;v_h).
\end{aligned}
\end{equation}
Decompose the error equation for velocity
\begin{equation}\label{deco_velo_error}
u(t_{n+1})- u^{n+1}_{h}=( u^{n+1}-\tilde{ u}^{n+1}_{h})+(\tilde{ u}^{n+1}_{h}- u^{n+1}_{h})= {\eta}^{n+1} + \phi^{n+1}_{h}.
\end{equation} where $\tilde{ u}^{n+1}_{h}$ is the best approximation of $ u(t^{n+1})$ in ${V}_{h}$.

Set $ v_{h}=I[\phi^{n+1}_{h}]$. Using the identity (\ref{eqn:theidentity_part2}) with $a=\phi^{n+1}_h$, $b=\phi^{n}_h$, $c=\phi^{n-1}_h$, (\ref{deco_velo_error}), and applying $(\lambda_h,\nabla \cdot \phi_h)=0$ for all $\lambda_h \in V^h$, equation ($\ref{subtract_first_result}$) can be written
\begin{equation}\label{subtract_result}
\begin{aligned}
&\dfrac{1}{4 \Delta t}(\|\phi^{n+1}_{h}\|^2+\|2\phi^{n+1}_{h}-\phi^{n}_{h}\|^2+\|\phi^{n+1}_{h}-\phi^{n}_{h}\|^2)\\
&-\dfrac{1}{4 \Delta t}(\|\phi^{n}_{h}\|^2+\|2\phi^{n}_{h}-\phi^{n-1}_{h}\|^2+\|\phi^{n}_{h}-\phi^{n-1}_{h}\|^2)\\
&+\dfrac{3}{4 \Delta t}\|\phi^{n+1}_{h}-2\phi^{n}_{h}+\phi^{n-1}_{h}\|^2
+\nu \|\nabla I[\phi^{n+1}_{h}]\|^2\\
&=-\left( \frac{D[\eta^{n+1}]}{\Delta t},I[\phi^{n+1}_{h}]  \right) -b\left(I[\phi^{n+1}_{h}],I[u(t^{n+1})], I[\phi^{n+1}_{h}]\right) \\
&-b\left( I[u^{n+1}_{h}], I[\eta^{n+1}], I[\phi^{n+1}_{h}]\right) -b\left(  I[\eta^{n+1}],I[u(t^{n+1})], I[\phi^{n+1}_{h}]\right)\\
&+\left( {p}(t^{n+1})-{\lambda}^{n+1}_h, \nabla \cdot I[\phi^{n+1}_{h}]\right)-\nu \left( \nabla I[\eta^{n+1}], \nabla I[\phi^{n+1}_{h}] \right) \\
&+ \tau^{n+1}(u,p;I[\phi^{n+1}_{h}]).
\end{aligned}
\end{equation}

The next step in the proof is to bound all the terms on the right hand side of $(\ref{subtract_result})$ and absorb terms into the left hand side. For arbitrary $\varepsilon>0$, the first term on the right hand side of (\ref{subtract_result}) is bounded in the following way,
\begin{equation}\label{bound_first_term}
\begin{aligned}
&-\left( \frac{D[\eta^{n+1}]}{\Delta t},  I[\phi^{n+1}_{h}]\right)
\leq \frac{1}{4\varepsilon}\bigg\|\frac{D[\eta^{n+1}]}{\Delta t}\bigg\|^2_{-1}
+\varepsilon\|\nabla I[\phi^{n+1}_{h}]\|^2.
\end{aligned}
\end{equation}
The first nonlinear term can be bounded as 
\begin{equation}\label{bound_first_nonlinearterm}
\begin{aligned}
&-b\left(I[\phi^{n+1}_{h}],I[u(t^{n+1})], I[\phi^{n+1}_{h}]\right)\leq C\|I[\phi^{n+1}_{h}]\| 
\|I[u(t^{n+1})]\|_2 \|\nabla I[\phi^{n+1}_{h}]
\| \\
&\leq\frac{C^2}{4\varepsilon}\|I[\phi^{n+1}_{h}]\|^2 \|I[u(t^{n+1})]\|^2_2+\varepsilon\|\nabla I[\phi^{n+1}_{h}]\|^2.
\end{aligned}
\end{equation}
The second nonlinear term is estimated by rewriting it using (\ref{deco_velo_error}) as follows
\begin{equation}\label{bound_2nd_nonlinearterm}
\begin{aligned}
&-b\left( I[u^{n+1}_{h}], I[\eta^{n+1}], I[\phi^{n+1}_{h}]\right)
=-b\left( I[u(t^{n+1})],I[\eta^{n+1}],I[\phi^{n+1}_{h}]\right) \\
&+b\left( I[\eta^{n+1}],I[\eta^{n+1}],I[\phi^{n+1}_{h}]\right) 
+b\left( I[\phi^{n+1}_{h}],I[\eta^{n+1}],I[\phi^{n+1}_{h}] \right).
\end{aligned}
\end{equation}
then find bounds for all terms on the right hand side of (\ref{bound_2nd_nonlinearterm}). We bound the third nonlinear term in \eqref{subtract_result} the same way as the first nonlinear term in \eqref{bound_2nd_nonlinearterm}.
\begin{equation}\label{bound_2nd_nonlinearterm_1}
\begin{aligned}
&-b\left( I[u(t^{n+1})],I[\eta^{n+1}],I[\phi^{n+1}_{h}]\right) \\
&\leq C \|\nabla I[u(t^{n+1})]\| \|\nabla I[\eta^{n+1}]\| \|\nabla I[\phi^{n+1}_{h}]\| \\
&\leq \frac{C^2}{4\varepsilon}\| u\|^2_{\infty,1} \|\nabla I[\eta^{n+1}]\|^2+\varepsilon\|\nabla I[\phi^{n+1}_{h}]\|^2,
\end{aligned}
\end{equation}
and
\begin{equation}\label{bound_2nd_nonlinearterm_2}
\begin{aligned}
&b\left( I[\eta^{n+1}],I[\eta^{n+1}],I[\phi^{n+1}_{h}]\right) 
\leq \frac{C^2}{4\varepsilon}\|\nabla I[\eta^{n+1}]\|^4+\varepsilon \|\nabla I[\phi^{n+1}_{h}]\|^2.
\end{aligned}
\end{equation}
Next, we have 
\begin{equation}\label{bound_2nd_nonlinearterm_3}
\begin{aligned}
&b\left( I[\phi^{n+1}_{h}],I[\eta^{n+1}],I[\phi^{n+1}_{h}] \right)\\
&\leq C \|I[\phi^{n+1}_{h}]\|^{\frac{1}{2}} \|\nabla I[\phi^{n+1}_{h}]\|^{\frac{1}{2}} 
\|\nabla I[\eta^{n+1}]\| \|\nabla I[\phi^{n+1}_{h}]\| \\
&\leq C h^{\frac{-1}{2}} \|I[\phi^{n+1}_{h}]\| \|\nabla I[\eta^{n+1}]\| \|\nabla I[\phi^{n+1}_{h}]\| \\
&\leq C h^{\frac{1}{2}} \|I[\phi^{n+1}_{h}]\| \|I[u(t^{n+1})]\|_2 \|\nabla I[\phi^{n+1}_{h}]\| \\
&\leq \frac{C^2}{4\varepsilon} h \|I[\phi^{n+1}_{h}]\|^2 \|I[u(t^{n+1})]\|^2_2 + \varepsilon \|\nabla I[\phi^{n+1}_{h}]\|^2.
\end{aligned}
\end{equation}
The pressure can be bounded as follows
\begin{equation}\label{bound_pressure}
\begin{aligned}
&\left( {p}(t^{n+1})-{\lambda}^{n+1}_h, \nabla \cdot I[\phi^{n+1}_{h}]\right)
\leq \frac{C^2}{4\varepsilon} \|{p}(t^{n+1})-{\lambda}^{n+1}_h\|^2+\varepsilon \|\nabla I[\phi^{n+1}_{h}]\|^2.
\end{aligned}
\end{equation}
Then we can bound the term after the pressure,
\begin{equation}\label{bound_before_last_term}
\begin{aligned}
&-\nu \left( \nabla I[\eta^{n+1}], \nabla(I[\phi^{n+1}_{h}]
) \right) 
\leq \frac{C^2}{4\varepsilon} \|\nabla I[\eta^{n+1}]\|^2+\varepsilon \|\nabla I[\phi^{n+1}_{h}]\|^2.
\end{aligned}
\end{equation}
Next we will bound all components of the consistency error $\tau^{n+1}(u,p;I[\phi^{n+1}_{h}])$.
\begin{equation}\label{bound_last_term}
\begin{aligned}
&\left( \frac{D[u(t^{n+1})]}{\Delta t}- u_{t}{(t^{n+1})},I[\phi^{n+1}_{h}]\right) \\
&\leq C \|\frac{D[u(t^{n+1})]}{\Delta t}- u_{t}{(t^{n+1})}\| \|\nabla I[\phi^{n+1}_{h}]\| \\
&\leq \frac{C^2}{4\varepsilon} \|\frac{D[u(t^{n+1})]}{\Delta t}- u_{t}{(t^{n+1})}\|^2+\varepsilon \|\nabla I[\phi^{n+1}_{h}]\|^2.
\end{aligned}
\end{equation}
\begin{equation}\label{bound_last_term2}
\begin{aligned}
&\nu \left( \nabla (I[u(t^{n+1})] - u(t^{n+1})),  \nabla I[\phi^{n+1}_{h}]\right)\\ 
&\leq \frac{C^2}{4\varepsilon} \|\nabla (I[u(t^{n+1})] - u(t^{n+1}))\|^2+\varepsilon \|\nabla I[\phi^{n+1}_{h}]\|^2.
\end{aligned}
\end{equation}
The nonlinear term in $\tau^{n+1}(u,p;I[\phi^{n+1}_{h}])$ is then estimated as follows,
\begin{equation}
\begin{aligned}
&b\left(I[u(t^{n+1})],I[u(t^{n+1})],  I[\phi^{n+1}_{h}]\right)
- b(u(t^{n+1}),u(t^{n+1}),I[\phi^{n+1}_{h}])\\
&=b\left(I[u(t^{n+1})] -u(t^{n+1}) ,I[u(t^{n+1})],  I[\phi^{n+1}_{h}]\right) 
-b(u(t^{n+1}),I[u(t^{n+1})] -u(t^{n+1}) ,I[\phi^{n+1}_{h}])\\
&\leq C  \|\nabla (I[u(t^{n+1})] -u(t^{n+1}))\| \|\nabla I[\phi^{n+1}_{h}]\|\Big( \|\nabla I[u(t^{n+1})]\| + \|\nabla u(t^{n+1})\| \Big)\\
&\leq \frac{C^2}{4\varepsilon} \|\nabla (I[u(t^{n+1})] -u(t^{n+1}))\|^2 \Big( \|\nabla I[u(t^{n+1})]\|^2 + \|\nabla u(t^{n+1})\|^2 \Big)
+\varepsilon \|\nabla I[\phi^{n+1}_{h}]\|^2.
\end{aligned}
\end{equation}
Set $\varepsilon=\frac{\nu}{16}$. Using $(\ref{bound_first_term})$ to $(\ref{bound_before_last_term})$ in \eqref{subtract_result} yields
\begin{equation}\label{bound_after_set_epsilon}
\begin{aligned}
&\dfrac{1}{4 \Delta t}(\|\phi^{n+1}_{h}\|^2+\|2\phi^{n+1}_{h}-\phi^{n}_{h}\|^2+\|\phi^{n+1}_{h}-\phi^{n}_{h}\|^2)
+\frac{\nu}{4} \|\nabla I[\phi^{n+1}_{h}]\|^2\\
&-\dfrac{1}{4 \Delta t}(\|\phi^{n}_{h}\|^2+\|2\phi^{n}_{h}-\phi^{n-1}_{h}\|^2+\|\phi^{n}_{h}-\phi^{n-1}_{h}\|^2)
+\dfrac{3}{4 \Delta t}\|\phi^{n+1}_{h}-2\phi^{n}_{h}+\phi^{n-1}_{h}\|^2\\
&\leq C  \Big( \|\frac{D[\eta^{n+1}]}{\Delta t}\|^2_{-1} 
+ (1+h)\|I[\phi^{n+1}_{h}]\|^2 \|I[u(t^{n+1})]\|^2_2\\
&+\| u\|^2_{\infty,1} \|\nabla I[\eta^{n+1}]\|^2
+ \|\nabla I[\eta^{n+1}]\|^4 
+ \|{p}(t^{n+1})-{\lambda}^{n+1}_h\|^2 \\
&+  \|\nabla I[\eta^{n+1}]\|^2
+ \|\frac{D[u(t^{n+1})]}{\Delta t}- u_{t}{(t^{n+1})}\|^2 \\
& + \|\nabla (I[u(t^{n+1})] - u(t^{n+1}))\|^2  \\
&+ \|\nabla (I[u(t^{n+1})] -u(t^{n+1}))\|^2 
( \|\nabla I[u(t^{n+1})]\|^2 + \|\nabla u(t^{n+1})\|^2 )\Big).
\end{aligned}
\end{equation}
Let $\kappa=C \nu \| u\|^2_{\infty,2}(1+h)$. Assume $\Delta t <\frac{1}{\kappa}$, summing from $n=1$ to $n=N-1$ and applying the discrete Gronwall lemma we obtain
\begin{equation}\label{bound_after_grwonwall}
\begin{aligned}
&\|\phi^{N}_{h}\|^2+\|2\phi^{N}_{h}-\phi^{N-1}_{h}\|^2+\|\phi^{N}_{h}-\phi^{N-1}_{h}\|^2\\
&+\sum_{n=1}^{N-1} 3\|\phi^{n+1}_{h}-2\phi^{n}_{h}+\phi^{n-1}_{h}\|^2 
+\nu \Delta t \sum_{n=1}^{N-1} \|\nabla I[\phi^{n+1}_{h}]\|^2\\
&\leq e^{\Big( \frac{\Delta t \kappa(N-1)}{1-\Delta t \kappa}\Big) } \Big(\|\phi^{1}_{h}\|^2+\|2\phi^{1}_{h}-\phi^{0}_{h}\|^2+\|\phi^{1}_{h}-\phi^{0}_{h}\|^2 
+C\Delta t \sum_{n=1}^{N-1}\|\frac{D[\eta^{n+1}]}{\Delta t}\|^2_{-1}\\
&+C \Delta t \nu ( \| u\|^2_{\infty,1}+ 1)\sum_{n=1}^{N-1}\|\nabla I[\eta^{n+1}]\|^2
+C \Delta t \sum_{n=1}^{N-1} \|\nabla I[\eta^{n+1}]\|^4 \\
&+C \Delta t \sum_{n=1}^{N-1}\|{p}(t^{n+1})-{\lambda}^{n+1}_h\|^2
+C \Delta t \sum_{n=1}^{N-1}\|\frac{D[u(t^{n+1})]}{\Delta t}- u_{t}{(t^{n+1})}\|^2 \\ 
&+C  \Delta t \sum_{n=1}^{N-1} \|\nabla (I[u(t^{n+1})] -u(t^{n+1}))\|^2 \\
&+C  \Delta t \sum_{n=1}^{N-1} \|\nabla (I[u(t^{n+1})] -u(t^{n+1}))\|^2 
( \|\nabla I[u(t^{n+1})]\|^2 + \|\nabla u(t^{n+1})\|^2 )\Big).
\end{aligned}
\end{equation}
The first three terms can be bounded as
\begin{equation}\label{bound_after_grwonwall_first_term}
\begin{aligned}
&\|\phi^{1}_{h}\|^2+\|2\phi^{1}_{h}-\phi^{0}_{h}\|^2+\|\phi^{1}_{h}-\phi^{0}_{h}\|^2\\
&\leq C \Big( \|u(t_1)-u^1_h\|^2 +\|(u(t_0)-u^0_h)\|^2 
 \Big)
+ C h^{2k+2}\||u\||^2_{\infty,k+1}.
\end{aligned}
\end{equation}
We bound the fourth term in (\ref{bound_after_grwonwall}) as follows
\begin{equation}\label{bound_after_grwonwall_second_term}
\begin{aligned}
&\nu\Delta t \sum_{n=1}^{N-1}\|\frac{D[\eta^{n+1}]}{\Delta t}\|^2_{-1}
=\nu\Delta t \sum_{n=1}^{N-1} \|\frac{\frac{3}{2}(\eta^{n+1}-\eta^{n})-\frac{1}{2}(\eta^{n}-\eta^{n-1})}{\Delta t}\|^2_{-1}\\
&\leq C \sum_{n=0}^{N} \int^{t^{n+1}}_{t^{n-1}} \|\eta_t\|^2 ds
\leq C h^{2k+2} \| u_t\|^2_{2,k+1},
\end{aligned}
\end{equation}
and 
\begin{equation}\label{bound_after_grwonwall_third_term}
\begin{aligned}
&\Delta t ( \nu \| u\|^2_{\infty,1}+ \nu)\sum_{n=1}^{N-1}\|\nabla I[\eta^{n+1}]\|^2\\
&\leq
C\Delta t \nu(2  \| u\|^2_{\infty,1}+1)\max \left\lbrace \frac{9}{4},4,\frac{1}{4} \right\rbrace \sum_{n=1}^{N-1} 3\left( \|\nabla\eta^{n+1}\|^2+ \|\nabla\eta^{n}\|^2 + \|\nabla\eta^{n-1}\|^2 \right)\\
& \leq C \Delta t \sum_{n=0}^{N} h^{2k} \| u^{n+1}\|^2_{k+1}
=Ch^{2k}\||  u\||^2_{2,k+1}.
\end{aligned}
\end{equation}
Similarly to (\ref{bound_after_grwonwall_third_term}), we also have 
\begin{equation}\label{bound_after_grwonwall_4th_term}
\begin{aligned}
&\Delta t \sum_{n=1}^{N-1} \|\nabla I[\eta^{n+1}]\|^4 
\leq
C \Delta t \sum_{n=0}^{N} h^{4k} \| u^{t+1}\|^4_{k+1}
=Ch^{4k} \|| u\||^4_{4,k+1}.
\end{aligned}
\end{equation}
Observe that 
\begin{equation}\label{bound_after_grwonwall_6th_term}
\begin{aligned}
& \nu \Delta t \sum_{n=1}^{N}\|{p}(t^{n+1})-{\lambda}^{n+1}_h\|^2
\leq
C h^{2s+2} \||p\||^2_{2,s+1}.
\end{aligned}
\end{equation}
The terms from consistency error are bounded using Lemma \ref{consis_lemma}.
\begin{equation}\label{bound_after_grwonwall_5th_term}
\begin{aligned}
& \nu \Delta t \sum_{n=1}^{N-1}\|\frac{D[u(t^{n+1})]}{\Delta t}- u_{t}{(t^{n+1})}\|^2
=C \Delta t ^4 \sum_{n=0}^{N-1} \int_{t^{n-1}}^{t^{n+1}} \| u_{ttt}\|^2dt
=C \Delta t ^4 \| u_{ttt}\|^2_{2,0}.
\end{aligned}
\end{equation}
\begin{equation}
\begin{aligned}
&\nu \Delta t \sum_{n=1}^{N-1} \|\nabla (I[u(t^{n+1})] -u(t^{n+1}))\|^2 
\leq C \Delta t ^4 \sum_{n=1}^{N-1} \int_{t^{n-1}}^{t^{n+1}} \| \nabla u_{tt}\|^2dt
\leq C \Delta t ^4 \| \nabla u_{tt}\|^2_{2,0}.
\end{aligned}
\end{equation}
\begin{equation}
\begin{aligned}
&\nu \Delta t \sum_{n=1}^{N-1} \|\nabla (I[u(t^{n+1})] -u(t^{n+1}))\|^2 
( \|\nabla I[u(t^{n+1})]\|^2 + \|\nabla u(t^{n+1})\|^2 )\\
&\leq C \Delta t \sum_{n=1}^{N-1}( \|\nabla I[u(t^{n+1})]\|^2 + \|\nabla u(t^{n+1})\|^2 ) \Delta t^3 \int_{t^n}^{t^{n+1}} \| \nabla u_{tt}\|^2dt \\
& \leq C \Delta t ^4 \sum_{n=1}^{N-1} (\int_{t^{n-1}}^{t^{n+1}} \|\nabla I[u(t^{n+1})]\|^4
+ \|\nabla u(t^{n+1})\|^4 + \| \nabla u_{tt}\|^4dt) \\
&\leq C \Delta t^4 ( \||\nabla u \||^4_{4,0} + \| \nabla u_{tt}\|^4_{4,0}).
\end{aligned}
\end{equation}
Combining $(\ref{bound_after_grwonwall_first_term})$ - $(\ref{bound_after_grwonwall_6th_term})$ gives
\begin{equation}\label{final_result}
\begin{aligned}
&\|\phi^{N}_{h}\|^2+\|2\phi^{N}_{h}-\phi^{N-1}_{h}\|^2+\|\phi^{N}_{h}-\phi^{N-1}_{h}\|^2+\sum_{n=1}^{N-1} 3\|\phi^{n+1}_{h}-2\phi^{n}_{h}+\phi^{n-1}_{h}\|^2 \\
&+\nu \Delta t \sum_{n=1}^{N-1} \|\nabla I[\phi^{n+1}_{h}]\|^2\\
&\leq
C\Big(  \|u(t_1)-u^1_h\|^2 +\|(u(t_0)-u^0_h)\|^2 
+h^{2k+2}\||u\||^2_{\infty,k+1}  \\
&+h^{2k+2} \| u_t\|^2_{2,k+1}
+h^{2k}\||  u\||^2_{2,k+1}
+h^{4k} \|| u\||^4_{4,k+1}
+h^{2s+2} \||p\||^2_{2,s+1}
 \\
&+ \Delta t ^4 ( \| u_{ttt}\|^2_{2,0}
+ \|\nabla u_{tt}\|^2_{2,0}
+\||\nabla u \||^4_{4,0} + \| \nabla u_{tt}\|^4_{4,0})\Big).
\end{aligned}
\end{equation}
We add both sides of $(\ref{final_result})$ with 
\begin{equation}\label{extra_term}
\begin{aligned}
&\|{\eta}^{N}\|^2+\|2{\eta}^{N}-{\eta}^{N-1}\|^2+\|{\eta}^{N}-{\eta}^{N-1}\|^2+\sum_{n=1}^{N-1} 3\|{\eta}^{n+1}-2{\eta}^{n}+{\eta}^{n-1}\|^2 \\
&+\nu \Delta t \sum_{n=1}^{N-1} \|\nabla (\frac{3}{2}{\eta}^{n+1}-{\eta}^{n}+\frac{1}{2}{\eta}^{n-1})\|^2.
\end{aligned}
\end{equation}
and apply triangle inequality to get $(\ref{velocity_error})$. 
\end{proof}
\label{section6}

\section{Second Order Error Estimator\label{sec:num_ode}}
This section justifies the use of $EST_2$ as an error estimator for the second order approximation. A Taylor series calculation shows that the second order approximation $y_{(2)}^{n+1}$ in Algorithm \ref{alg:vsvo_themethod} has the local truncation error (LTE) (for constant stepsize)
\begin{equation}\notag
LTE = -\Delta t^3\left(\frac{1}{3}y''' + \frac{1}{2}f_y y''\right) + \mathcal{O}(\Delta t^4).
\end{equation}
 Consider the addition of a second time filter,
 \begin{equation}
\begin{array}{ccc}
\text{Step 1} & : & \frac{y_{n+1}^{(1)}-y^{n}}{\Delta t}=f(t_{n+1},y^{n+1}_{(1)}), \\ 
\text{Step 2} & : & y_{(2)}^{n+1}=y^{n+1}_{(1)}-\frac{1}{3}\left\{
y^{n+1}_{(1)}-2y^{n}+y^{n-1}\right\}\\
\text{Step 3} & : & y_{n+1}=y_{(2)}^{n+1}-\frac{2}{11}\left\{
y_{(2)}^{n+1}-3y^{n}+3y^{n-1} - y^{n-2}\right\}
\end{array}
\label{eq:methodConstantTimestepMoreFilter}
\end{equation}%
Another Taylor series calculation shows that the induced method has the LTE of
\begin{equation}\notag
LTE = -\Delta t^3\frac{1}{2}f_y y'' + \mathcal{O}(\Delta t^4),
\end{equation}
Thus, $y_{n+1}$ yields a more accurate (still second order) approximation, and $$EST_2=y_{(2)}^{n+1}-y_{n+1}=\frac{2}{11}\left\{
y_{n+1}^{(2)}-3y^{n}+3y^{n-1} - y^{n-2}\right\}$$ gives an estimate for the error of $y_{n+1}$. This is extended to variable stepsize using Newton interpolation, and written with stepsize ratios in Algorithm \ref{alg:vsvo_themethod}.

This is a nonstandard approach since one would normally use a higher order approximation to estimate the error. However, this is simple since it requires no additional function evaluations or Jacobians, and does not require solving a system of equations. Interestingly, \eqref{eq:methodConstantTimestepMoreFilter} remains energy stable, and could be useful as a standalone constant stepsize method.
\end{document}